\documentclass[12pt,reqno]{amsart}
\usepackage{amsmath,amssymb,amsthm,amscd,amsfonts,geometry,color}
\newtheorem{theorem}{Theorem}[section]

\newtheorem{lemma}[theorem]{Lemma}
\newtheorem{proposition}[theorem]{Proposition}

\newtheorem{thmA}{Theorem~A}
 
\newtheorem{thmB}{Theorem~B}
 
\newtheorem{thmC}{Theorem~C}

\theoremstyle{definition}

\newcommand{\R}{\mathbb{R}}                       
\newcommand{\N}{\mathbb{N}}                       

\newcommand{\0}{\mathbf{0}}                       

\newcommand{\fin}{\operatorname{Fin}}             
\newcommand{\doubl}{\operatorname{D}}             

\newcommand{\conti}{\operatorname{C}}             
\newcommand{\uc}{\operatorname{UC}}             
\newcommand{\lip}{\operatorname{LIP}}             
\newcommand{\lipc}{\operatorname{lip}}            
\newcommand{\pseudo}{\operatorname{PM}}             
\newcommand{\metr}{\operatorname{M}}             
\newcommand{\adm}{\operatorname{AM}}             
\newcommand{\upm}{\operatorname{UPM}}             
\newcommand{\lpm}{\operatorname{LPM}}             

\begin{document}

\title[Banach-Stone type theorems on pseudometrics]{Banach-Stone type theorems on uniformly continuous and lipschitz continuous pseudometrics}
\author{Katsuhisa Koshino}
\address[Katsuhisa Koshino]{Faculty of Engineering, Kanagawa University, 3-27-1 Rokkakubashi, Kanagawa-ku, Yokohama-shi, 221-8686, Japan}
\email{ft160229no@kanagawa-u.ac.jp}
\maketitle

\begin{abstract}
In this paper, we shall establish Banach-Stone type theorems on spaces of uniformly continuous and lipschitz continuous pseudometrics.
\end{abstract}

\section{Introduction}

The Banach-Stone theorem \cite{Banac,MHSto} asserts that metric structures of spaces of continuous functions determine the topologies on their underlying spaces.
This is one of the most important results and has many developments in functional analysis, see \cite{FJ,GJ1,CJV} for the history and some variations.
Throughout the paper, spaces $Z = (Z,\mathcal{U}_Z)$ are uniform spaces,
 which are topologized by their uniformities $\mathcal{U}_Z$.
An isometry is a surjective isometry,
 and the symbols $\R$ and $\N$ stand for the set of real numbers and the one of positive integers, respectively.
We let $\conti(Z) = (\conti(Z),\|\cdot\|)$ be the space of real-valued bounded continuous functions on $Z$ with the sup-norm $\|\cdot\|$.
In the case where $Z = (Z,d_Z)$ is a metric space with an admissible metric $d_Z$, let $\uc(Z,d_Z) \subset \conti(Z)$ be the subspace consisting of uniformly continuous functions.
In the paper \cite{Her} (cf.~\cite{AF}), a Banach-Stone type theorem, which states that the metric structure of $\uc(Z,d_Z)$ characterizes the uniform structure of $Z$,
 was obtained as follows:

\begin{theorem}
Suppose that $X = (X,d_X)$ and $Y = (Y,d_Y)$ are complete metric spaces.
The following are equivalent:
\begin{enumerate}
 \item $X$ and $Y$ are uniformly homeomorphic;
 \item there exists an isometry $T : \uc(X,d_X) \to \uc(Y,d_Y)$.
\end{enumerate}
In this case, for each isometry $T : \uc(X,d_X) \to \uc(Y,d_Y)$, there exists a uniform homeomorphism $\phi : Y \to X$ and $\alpha \in \uc(Y,d_Y)$ with $\alpha(Y) \subset \{1,-1\}$ such that for any $f \in \uc(X,d_X)$ and for any $y \in Y$,
 $$T(f)(y) = \alpha(y)f(\phi(y)).$$
\end{theorem}

In this paper, we shall establish Banach-Stone type theorems on spaces of uniformly continuous and lipschitz continuous metrics.
Let $\pseudo(Z) \subset \conti(Z^2)$ be the space consisting of pseudometrics on a space $Z$,
 and let $\adm(Z) \subset \pseudo(Z)$ be the subspace of admissible metrics.
Set
\begin{align*}
 {\rm Pc}(Z) = \{d \in \pseudo(Z) \mid &\text{ there are a doubleton} \{z,w\} \subset Z \text{ and a compact subset } K \subset Z\\
 &\ \ \ \ \text{ such that } d(z,w) = \|d\| \text{ and if } d(x,y) = \|d\|,
 \text{ then } x, y \in K\},
\end{align*}
 $${\rm Pp}(Z) = \{d \in \pseudo(Z) \mid \text{ there uniquely exists } \{z,w\} \subset Z \text{ such that } d(z,w) = \|d\|\}.$$
Remark that ${\rm Pp}(Z) \subset {\rm Pc}(Z)$,
 and that $\pseudo(Z) = {\rm Pc}(Z)$ if $Z$ is compact.
We denote either ${\rm Pc}(Z)$ or ${\rm Pp}(Z)$ by ${\rm P}(Z)$.
M.E.~Shanks \cite{Sh} established a Banach-Stone type theorem on spaces of pseudometrics and admissible metrics of compact metrizable spaces,
 and the author \cite{Kos29} generalized his result and described isometries by using homeomorphisms as follows:

\begin{theorem}\label{isometry}
Let $X$ and $Y$ be metrizable spaces.
The following are equivalent:
\begin{enumerate}
 \item $X$ and $Y$ are homeomorphic;
 \item there exists an isometry $T : \pseudo(X) \to \pseudo(Y)$ with $T({\rm P}(X) \cap \adm(X)) = {\rm P}(Y) \cap \adm(Y)$;
 \item there exists an isometry $T : \adm(X) \to \adm(Y)$ with $T({\rm P}(X) \cap \adm(X)) = {\rm P}(Y) \cap \adm(Y)$;
 \item there exists an isometry $T : {\rm P}(X) \cap \adm(X) \to {\rm P}(Y) \cap \adm(Y)$.
\end{enumerate}
In this case, for each isometry $T : \pseudo(X) \to \pseudo(Y)$ with $T({\rm P}(X) \cap \adm(X)) = {\rm P}(Y) \cap \adm(Y)$, there is a homeomorphism $\phi : Y \to X$ such that for any $d \in \pseudo(X)$ and for any $x, y \in Y$,
 $$T(d)(x,y) = d(\phi(x),\phi(y)).$$
Except for the case where the cardinality of $X$ or $Y$ is equal to $2$, the homeomorphism $\phi$ can be chosen uniquely.
\end{theorem}

A pseudometric $d$ on $Z$ is \textit{uniformly continuous} with respect to a uniformity $\mathcal{U}_Z$ (respectively, an admissible metric $d_Z$) of $Z$ provided that for each $\epsilon > 0$, there is an entourage $U \in \mathcal{U}_Z$ (respectively, a positive number $\delta > 0$) such that if $(x,y) \in U$ (respectively, $d_Z(x,y) < \delta$),
 then $d(x,y) < \epsilon$.
Furthermore, we say that $d$ is \textit{lipschitz continuous} with respect to $d_Z$ if there exists $k > 0$ such that for every pair $(x,y) \in Z^2$ with $x \neq y$, $d(x,y) \leq kd_Z(x,y)$,
 where
 $$\lipc(d) = \sup_{(x,y) \in Z^2 \text{ with } x \neq y} \frac{d(x,y)}{d_Z(x,y)}$$
 is called the \textit{lipschitz constant} of $d$.
Let $\upm(Z,\mathcal{U}_Z)$ (respectively, $\upm(Z,d_Z)$) be the space consisting of bounded uniformly continuous pseudometrics on $Z$ with respect to $\mathcal{U}_Z$ (respectively, $d_Z$),
 and let $\lpm(Z,d_Z)$ be the one consisting of bounded lipschitz continuous pseudometrics on $Z$ with respect to $d_Z$.
Moreover, for positive number $k > 0$, put
 $$\lpm_k(Z,d_Z) = \{d \in \lpm(Z,d_Z) \mid \lipc(d) < k\},$$
 $$\overline{\lpm}_k(Z,d_Z) = \{d \in \lpm(Z,d_Z) \mid \lipc(d) \leq k\}.$$
Clearly, $\lpm_k(Z,d_Z) \subset \overline{\lpm}_k(Z,d_Z)$.
By ${\rm U}(Z,d_Z)$, denote the space consisting of $d \in \pseudo(Z)$ which satisfies the following:
 \begin{itemize}
  \item For every $\epsilon > 0$, there exist a finite set $F \subset \uc(Z,d_Z)$ and $\delta > 0$ such that if $|f(x) - f(y)| < \delta$ for any $f \in F$,
  then $d(x,y) < \epsilon$.
 \end{itemize}
We shall show the following:

\begin{thmA}
Let $X = (X,\mathcal{U}_X)$ and $Y = (Y,\mathcal{U}_Y)$ be spaces.
The following are equivalent:
\begin{enumerate}
 \item $X$ and $Y$ are uniformly homeomorphic;
 \item there is an isometry $T : \upm(X,\mathcal{U}_X) \to \upm(Y,\mathcal{U}_Y)$ with $T({\rm P}(X) \cap \upm(X,\mathcal{U}_X)) = {\rm P}(Y) \cap \upm(Y,\mathcal{U}_X)$.
\end{enumerate}
Moreover, if $X = (X,d_X)$ and $Y = (Y,d_Y)$ be complete metric spaces,
 the following is also equivalent to the above.
\begin{enumerate}
\setcounter{enumi}{2}
 \item there exists an isometry $T : {\rm U}(X,d_X) \to {\rm U}(Y,d_X)$.
\end{enumerate}
In these cases, for every isometry $T : \upm(X,\mathcal{U}_X) \to \upm(Y,\mathcal{U}_Y)$ with $T({\rm P}(X) \cap \upm(X,\mathcal{U}_X)) = {\rm P}(Y) \cap \upm(Y,\mathcal{U}_X)$ (respectively, $T : {\rm U}(X,d_X) \to {\rm U}(Y,d_Y)$), there exists a uniform homeomorphism $\phi : Y \to X$ such that for any $d \in \upm(X,\mathcal{U}_X)$ (respectively, $d \in {\rm U}(X,d_X)$) and for any $x, y \in Y$,
 $$T(d)(x,y) = d(\phi(x),\phi(y)).$$
Except for the case where the cardinality of $X$ or $Y$ is equal to $2$, the map $\phi$ can be chosen uniquely.
\end{thmA}

Furthermore, we will prove the following theorems on lipschitz continuous pseudometrics:

\begin{thmB}
Suppose that $X = (X,d_X)$ and $Y = (Y,d_Y)$ are bounded metric spaces.
The following are equivalent:
\begin{enumerate}
 \item $X$ and $Y$ are bi-lipschitz;
 \item there exists an isometry $T : \lpm(X,d_X) \to \lpm(Y,d_Y)$ with $T({\rm P}(X) \cap \lpm(X,d_X)) = {\rm P}(Y) \cap \lpm(Y,d_Y)$.
\end{enumerate}
In this case, for each isometry $T : \lpm(X,d_X) \to \lpm(Y,d_Y)$ with $T({\rm P}(X) \cap \lpm(X,d_X)) = {\rm P}(Y) \cap \lpm(Y,d_Y)$, there exists a bi-lipschitz homeomorphism $\phi : Y \to X$ such that for all $d \in \lpm(X,d_X)$ and for all $x, y \in Y$,
 $$T(d)(x,y) = d(\phi(x),\phi(y)).$$
Except for the case that the cardinality of $X$ or $Y$ is equal to $2$, the map $\phi$ can be chosen uniquely.
\end{thmB}

\begin{thmC}
Let $X = (X,d_X)$ and $Y = (Y,d_Y)$ be bounded metric spaces,
 and let $k > 0$.
The following are equivalent:
\begin{enumerate}
 \item $X$ and $Y$ are isometric;
 \item there exists an isometry $T : \overline{\lpm}_k(X,d_X) \to \overline{\lpm}_k(Y,d_Y)$ with $T(\lpm_k(X,d_X)) = \lpm_k(Y,d_Y)$ and $T({\rm P}(X) \cap \lpm_k(X,d_X)) = {\rm P}(Y) \cap \lpm_k(Y,d_Y)$;
 \item there exists an isometry $T : \lpm_k(X,d_X) \to \lpm_k(Y,d_Y)$ with $T({\rm P}(X) \cap \lpm_k(X,d_X)) = {\rm P}(Y) \cap \lpm_k(Y,d_Y)$.
\end{enumerate}
In this case, for each isometry $T : \overline{\lpm}_k(X,d_X) \to \overline{\lpm}_k(Y,d_Y)$ with $T(\lpm_k(X,d_X)) = \lpm_k(Y,d_Y)$ and $T({\rm P}(X) \cap \lpm_k(X,d_X)) = {\rm P}(Y) \cap \lpm_k(Y,d_Y)$, there is an isometry $\phi : Y \to X$ such that for any $d \in \overline{\lpm}_k(X,d_X)$ and for any $x, y \in Y$,
 $$T(d)(x,y) = d(\phi(x),\phi(y)).$$
Except for the case where the cardinality of $X$ or $Y$ is equal to $2$, the isometry $\phi$ can be chosen uniquely.
\end{thmC}

\section{Uniformities and pseudometrics}\label{unif.metr.}

Before our main argument, let us review the study on uniformities and pseudometrics.
For the definition and basic theory of uniform spaces, refer to Chapter~8 of \cite{En}.
A space $Z = (Z,\mathcal{U}_Z)$ induces the following topology,
 which is Tychonoff,\footnote{Conversely, a Tychonoff space admits a uniformity compatible with its original topology, see \cite[Theorem~8.1.20]{En}.} (cf.~\cite[Corollary~8.1.13]{En}):
 $$\mathcal{T}(\mathcal{U}_Z) = \{U \subset Z \mid \text{for any } x \in U, \text{ there is } V \in \mathcal{U}_Z \text{ such that } \{y \in Z \mid (x,y) \in V\} \subset U\}.$$
Uniformity is a generalization of metric.
It is known that for a metric space $Z = (Z,d_Z)$, its metric $d_Z$ can generate a uniformity as follows (cf.~\cite[Proposition~8.1.18]{En}):
 $$\mathcal{U}_{d_Z} = \{U \subset Z^2 \mid U \text{ is an entourage and } \{(x,y) \in Z^2 \mid d_Z(x,y) < \epsilon\} \subset U \text{ for some } \epsilon > 0\}.$$
The uniform continuity of a pseudometric $d$ on $Z$ with respect to $\mathcal{U}_{d_Z}$ coincides with the one with respect to $d_Z$, that is, $\upm(Z,d_Z) = \upm(Z,\mathcal{U}_{d_Z})$.
We similarly define the function spaces $\uc(Z,\mathcal{U}_Z)$ and ${\rm U}(Z,\mathcal{U}_Z)$,
 and notice that for a metric space $Z = (Z,d_Z)$, $\uc(Z,d_Z) = \uc(Z,\mathcal{U}_{d_Z})$ and ${\rm U}(Z,d_Z) = {\rm U}(Z,\mathcal{U}_{d_Z})$.
As is easily observed,
 a uniformly continuous pseudometric $d$ is continuous,
 so $\upm(Z,\mathcal{U}_Z) \subset \pseudo(Z)$ (cf.~\cite[Proposition~8.1.9]{En}).

A space $Z = (Z,\mathcal{U}_Z)$ is \textit{totally bounded} if for each entourage $U \in \mathcal{U}_Z$, there is a finite subset $A \subset Z$ such that for every $x \in Z$, there exists $y \in A$ such that the pair $(x,y) \in U$.
We say that $Z$ is \textit{complete} if the following condition holds:
\begin{itemize}
 \item Let $\mathcal{A}$ be any family of closed sets in $Z$ which has the finite intersection property.
 If for each $U \in \mathcal{U}_Z$, there exists $A \in \mathcal{A}$ such that $A^2 \subset U$,
 then $\bigcap \mathcal{A} \neq \emptyset$.
\end{itemize}
Recall that in the case where $\mathcal{U}_Z = \mathcal{U}_{d_Z}$ for some metric $d_Z$ on $Z$, the total boundedness (respectively, completeness) on a uniform space coincides with the one on a metric space, see Propositions~8.3.1 and 8.3.5 of \cite{En}.
A space $Z$ is compact if and only if it is totally bounded and complete (cf.~Theorem~8.3.16 of \cite{En}),
 and then there exists the only one uniformity compatible with the original topology on $Z$ (cf.~\cite[Theorem~8.3.13]{En}).

For a subset $W$ of a space $Z = (Z,\mathcal{U}_Z)$, we induce the relative uniformity on $W$ by
 $$\mathcal{U}_Z|_W = \{U \cap W^2 \mid U \in \mathcal{U}_Z\}.$$
Obviously, for each $d \in \upm(Z,\mathcal{U}_Z)$, the restriction $d|_{W^2} \in \upm(W,\mathcal{U}_Z|_W)$.
It is known that for every space $Z$, there exists exactly one complete space $\overline{Z} = (\overline{Z},\overline{\mathcal{U}_Z})$ such that $Z$ is uniformly homeomorphic to some dense subspace $W = (W,\overline{\mathcal{U}_Z}|_W)$ in $\overline{Z}$,
 where we call $\overline{Z}$ the \textit{completion} of $Z$ (refer to \cite[Theorem~8.3.12]{En}).
By virtue of \cite[Exercise~8.1.F and Problem~8.5.7.(a)]{En}, we can obtain the coarsest uniformity $(\mathcal{U}_Z)_0$ among uniformities $\mathcal{V}_Z$ on $Z$ satisfying that all functions belonging in $\uc(Z,\mathcal{U}_Z)$ are also uniformly continuous with respect to $\mathcal{V}_Z$.
Then $(\mathcal{U}_Z)_0$ is the finest uniformity that is totally bounded and coarser than $\mathcal{U}_Z$.
The completion $uZ = (\overline{Z},\overline{(\mathcal{U}_Z)_0})$ of $Z = (Z,(\mathcal{U}_Z)_0)$ is called \textit{the Smirnov (or Samuel) compactification} of $Z$.
Given a finite subset $F \subset \uc(Z,\mathcal{U}_Z)$, we define a function $d(F) : Z^2 \to \R$ by
 $$d(F)(x,y) = \max\{|f(x) - f(y)| \mid f \in F\}.$$
Observe that $d(F) \in \upm(Z,\mathcal{U}_Z)$.
Set
 $$\mathcal{B}_Z = \{\{(x,y) \in Z^2 \mid d(F)(x,y) < 2^{-i}\} \mid F \text{ is a finite set in } \uc(Z,\mathcal{U}_Z) \text{ and } i \in \N\}.$$
Then we have the following:

\begin{proposition}\label{base}
For a space $Z = (Z,\mathcal{U}_Z)$,
 $$(\mathcal{U}_Z)_0 = \{U \subset Z^2 \mid U \text{ is an entourage and contains some } B \in \mathcal{B}_Z\}.\footnote{Refer to Example~8.1.19 of \cite{En} for a similar uniformity.}$$
\end{proposition}

\begin{proof}
As is easily observed,
 the right hand side is a uniformity on $Z$.
First, to prove that $(\mathcal{U}_Z)_0$ is contained in the right hand side, fix $f \in \uc(Z,\mathcal{U}_Z)$ arbitrarily.
For each positive integer $i \in \N$, take the following entourage:
 $$B = \{(x,y) \in Z^2 \mid d(\{f\})(x,y) < 2^{-i}\} \in \mathcal{B}_Z.$$
Then for all pairs $(x,y) \in B$,
 $$|f(x) - f(y)| = d(\{f\})(x,y) < 2^{-i},$$
 which implies that $f$ is uniformly continuous with respect to the right hand side.
By the minimality of $(\mathcal{U}_Z)_0$, the right hand side contains $(\mathcal{U}_Z)_0$.

Next, we will show the converse inclusion.
Let any entourage $U$ in the right hand side,
 so we can find a finite subset $F \subset \uc(Z,\mathcal{U}_Z)$ and a positive integer $i \in \N$ so that if $d(F)(x,y) < 2^{-i}$,
 then $(x,y) \in U$.
Since each $f \in F$ is uniformly continuous with respect to $(\mathcal{U}_Z)_0$,
 there is $V \in (\mathcal{U}_Z)_0$ such that if $(x,y) \in V$,
 then $|f(x) - f(y)| < 2^{-i}$ for all $f \in F$,
 that is $d(F)(x,y) < 2^{-i}$.
Then $V \subset U$,
 and hence $U \in (\mathcal{U}_Z)_0$.
The right hand side is contained in $(\mathcal{U}_Z)_0$.
The proof is completed.
\end{proof}

We can describe the space ${\rm U}(Z,\mathcal{U}_Z)$ as follows:

\begin{proposition}\label{unif.}
For a space $Z = (Z,\mathcal{U}_Z)$,
 $${\rm U}(Z,\mathcal{U}_Z) = \{d \in \upm(Z,\mathcal{U}_Z) \mid d \text{ is uniformly continuous with respect to } (\mathcal{U}_Z)_0\}.$$
\end{proposition}

\begin{proof}
For each $d \in {\rm U}(Z,\mathcal{U}_Z)$ and for each $\epsilon > 0$, we can choose a finite set $F$ in $\uc(Z,\mathcal{U}_Z)$ and a positive number $\delta > 0$ so that if $|f(x) - f(y)| < \delta$ for any $f \in F$,
 then $d(x,y) < \epsilon$.
Let
 $$B = \{(x,y) \in Z^2 \mid d(F)(x,y) < 2^{-i}\} \in \mathcal{B}_Z,$$
 where $i \in \N$ satisfies that $2^{-i} \leq \delta$,
 so $B \in (\mathcal{U}_Z)_0$ by Proposition~\ref{base}.
Obviously, for any $(x,y) \in B$ and any $f \in F$,
 $$|f(x) - f(y)| \leq d(F)(x,y) < 2^{-i} \leq \delta,$$
 which implies that $d(x,y) < \epsilon$.
Hence $d$ is uniformly continuous with respect to $(\mathcal{U}_Z)_0$.

Fix any uniformly continuous pseudometric $d$ with respect to $(\mathcal{U}_Z)_0$.
Then for every $\epsilon > 0$, there is $U \in (\mathcal{U}_Z)_0$ such that if $(x,y) \in U$,
 $d(x,y) < \epsilon$.
According to Proposition~\ref{base}, for the entourage $U$, we can take a finite subset $F \subset \uc(Z,\mathcal{U}_Z)$ and $i \in \N$ which fulfill the condition that $d(F)(x,y) < 2^{-i}$ implies that $(x,y) \in U$.
Then if $|f(x) - f(y)| < 2^{-i}$ for all functions $f \in F$,
 then $d(F)(x,y) < 2^{-i}$,
 and hence $(x,y) \in U$.
Therefore $d(x,y) < \epsilon$ and $d \in {\rm U}(Z,\mathcal{U}_Z)$.
We complete the proof.
\end{proof}

Remark that the spaces
 $${\rm U}(Z,\mathcal{U}_Z) \subset \upm(Z,\mathcal{U}_Z) \subset \pseudo(Z)$$
 are closed cones in the Banach space $\conti(Z^2)$,
 which are containing the zero function.
Their sup-metrics induced by $\|\cdot\|$ are complete.
When $Z$ is totally bounded,
 $\mathcal{U}_Z = (\mathcal{U}_Z)_0$,
 and hence $\upm(Z,\mathcal{U}_Z) = {\rm U}(Z,\mathcal{U}_Z)$ by Proposition~\ref{unif.}.
Moreover, due to \cite[Theorem~8.3.13]{En}, $\pseudo(Z) = \upm(Z,\mathcal{U}_Z)$ if $Z$ is compact.
When $Z = (Z,d_Z)$ is a metric space,
 in $\upm(Z,d_Z)$, the space $\lpm(Z,d_Z)$ is a cone,
 and the spaces $\lpm_k(Z,d_Z)$ and $\overline{\lpm}_k(Z,d_Z)$ are convex sets,
 all of which are also contain the zero function.
The space $\overline{\lpm}_k(Z,d_Z)$ is a closed subset of $\pseudo(Z)$.
In fact, for each $d \in \pseudo(Z) \setminus \overline{\lpm}_k(Z,d_Z)$, there is a pair $(x,y) \in Z^2$ of distinct points such that $d(x,y) > kd_Z(x,y)$.
Then for all pseudometrics $\rho \in \pseudo(Z)$ with
 $$\|\rho - d\| < d(x,y) - kd_Z(x,y),$$
 $\rho(x,y) > kd_Z(x,y)$,
 which means that $\rho \in \pseudo(Z) \setminus \overline{\lpm}_k(Z,d_Z)$,
 and that $\overline{\lpm}_k(Z,d_Z)$ is closed in $\pseudo(Z)$.
Observe that $\lpm_k(Z,d_Z)$ is dense in $\overline{\lpm}_k(Z,d_Z)$.
Indeed, for every $d \in \overline{\lpm}_k(Z,d_Z) \setminus \lpm_k(Z,d_Z)$ and every $\epsilon > 0$, taking a positive number $\delta < \min\{1,\epsilon/\|d\|\}$, we have that for all $x, y \in Z$ with $x \neq y$,
 $$\frac{(1 - \delta)d(x,y)}{d_Z(x,y)} \leq (1 - \delta)\lipc(d) < \lipc(d) = k.$$
Hence $(1 - \delta)d \in \lpm_k(Z,d_Z)$.
Furthermore,
 $$\|d - (1 - \delta)d\| = \|\delta d\| = \delta\|d\| < \frac{\epsilon}{\|d\|} \cdot \|d\| = \epsilon.$$

The Hausdorff metric extension theorem \cite{Hau} plays an important role in the previous paper \cite{Kos29} to construct some nice metrics instead of Urysohn's lemma frequently used in proving various Banach-Stone type theorems.
We can obtain the uniformly continuous and lipschitz continuous pseudometric versions of it, refer to Lemma~1.4 of \cite{Isb1} (cf.~Problem~8.5.6(a) of \cite{En}) and Theorem~4.1 of \cite{BBST}.

\begin{theorem}\label{ext.}
Suppose that $Z = (Z,\mathcal{U}_Z)$ is a space and $A \subset Z$ is a closed subset.
For each $d \in \upm(A,\mathcal{U}_Z|_A)$, there is $\tilde{d} \in \upm(Z,\mathcal{U}_Z)$ such that $\tilde{d}|_{A^2} = d$ and $\|\tilde{d}\| = \|d\|$.
Moreover, if $Z = (Z,d_Z)$ is a metric space,
 then for each $d \in \lpm(A,d_Z|_{A^2})$, there is $\tilde{d} \in \lpm(Z,d_Z)$ such that $\tilde{d}|_{A^2} = d$, $\|\tilde{d}\| = \|d\|$ and $\lipc(\tilde{d}) = \lipc(\tilde{d}|_{A^2})$.
\end{theorem}

\section{The peaking function argument}\label{peak}

Our strategy to prove Theorems~A, B and C is the peaking function argument,
 which traces its history back to Stone's method.
From now on, let $X = (X,\mathcal{U}_X)$ and $Y = (Y,\mathcal{U}_Y)$ be non-degenerate spaces,
 and let a pair $\metr(X)$ and $\metr(Y)$ be $\upm(X,\mathcal{U}_X)$ and $\upm(Y,\mathcal{U}_Y)$, in the case that $X = (X,d_X)$ and $Y = (Y,d_Y)$ are metric spaces, or $\lpm(X,d_X)$ and $\lpm(Y,d_Y)$, or $\lpm_k(X,d_X)$ and $\lpm_k(Y,d_Y)$,
 where $k > 0$ is arbitrary.
Suppose that $T : \metr(X) \to \metr(Y)$ is an isometry.
When the pair $\metr(X)$ and $\metr(Y)$ is $\upm(X,\mathcal{U}_X)$ and $\upm(Y,\mathcal{U}_Y)$ or $\lpm(X,d_X)$ and $\lpm(Y,d_Y)$,
 the results of this section can be obtained in the similar way as in Section~3 and 4 of \cite{Kos29}.
Therefore we will investigate mainly the case where $\metr(X) = \lpm_k(X,d_X)$ and $\metr(Y) = \lpm_k(Y,d_Y)$, respectively.

\begin{proposition}\label{norm}
The isometry $T$ is norm-preserving,
 that is, for any $d \in \metr(X)$ and any $\rho \in \metr(Y)$, $\|T(d)\| = \|d\|$ and $\|T^{-1}(\rho)\| = \|\rho\|$.
\end{proposition}

\begin{proof}
We will prove only the case that $\metr(X) = \lpm_k(X,d_X)$ and $\metr(Y) = \lpm_k(Y,d_Y)$.
This proposition follows from that $T(\0_{X^2}) = \0_{Y^2}$ and $T^{-1}(\0_{Y^2}) = \0_{X^2}$ immediately,
 where $\0_{X^2}$ and $\0_{Y^2}$ are the zero functions on $X^2$ and $Y^2$ respectively.
Similar to \cite[Lemma~2.3]{HMM}, for every $d \in \lpm_k(X,d_X)$, if $\|T(\0_{X^2})\| < \|d\|$,
 then $\|d\| \leq \|T(d)\|$,
 and for every $\rho \in \lpm_k(Y,d_Y)$, if $\|T^{-1}(\0_{Y^2})\| < \|\rho\|$,
 then $\|\rho\| \leq \|T^{-1}(\rho)\|$.
Supposing that $T^{-1}(\0_{Y^2}) \neq \0_{X^2}$, we can get that
 $$\|T^{-1}(\0_{Y^2})\| > 0 = \|\0_{Y^2}\| = \|T(T^{-1}(\0_{Y^2}))\|.$$
Thus $\|T^{-1}(\0_{Y^2})\| \leq \|T(\0_{X^2})\|$ and $0 < \lipc(T(\0_{X^2})) < k$.
Let
 $$t = \frac{\lipc(T(\0_{X^2})) + k}{2\lipc(T(\0_{X^2}))} > 1,$$
 so $tT(\0_{X^2}) \in \lpm_k(Y,d_Y)$ and
 $$\|T^{-1}(\0_{Y^2})\| \leq \|T(\0_{X^2})\| < t\|T(\0_{X^2})\| = \|tT(\0_{X^2})\|.$$
Hence $\|tT(\0_{X^2})\| \leq \|T^{-1}(tT(\0_{X^2}))\|$.
Since $T$ is an isometry,
 \begin{align*}
  t\|T(\0_{X^2})\| &> (t - 1)\|T(\0_{X^2})\| = \|(t - 1)T(\0_{X^2})\| = \|tT(\0_{X^2}) - T(\0_{X^2})\|\\
  &= \|T^{-1}(tT(\0_{X^2})) - T^{-1}(T(\0_{X^2}))\| = \|T^{-1}(tT(\0_{X^2}))\|\\
  &\geq \|tT(\0_{X^2})\| = t\|T(\0_{X^2})\|,
 \end{align*}
 which is a contradiction.
It follows that $T^{-1}(\0_{Y^2}) = \0_{X^2}$,
 and thus $T(\0_{X^2}) = \0_{Y^2}$.
The rest of proof shall be left to the readers.
\end{proof}

Combining the above proposition and the same argument as the Mazur-Ulam theorem \cite{MU} (cf.~Theorem~1.3.5 of \cite{FJ}), we can establish the following:

\begin{proposition}\label{scalar}
The isometry $T$ is scalar-preserving,
 that is, for every $d \in \metr(X)$ and every $t \geq 0$ with $td \in \metr(X)$, $T(td) = tT(d)$.
\end{proposition}

For a space $Z$ and a pseudometric $d \in \pseudo(Z)$, let
 $$B_d(z,r) = \{w \in Z \mid d(z,w) < r\} \text{ and } \overline{B}_d(z,r) = \{w \in Z \mid d(z,w) \leq r\}.$$
The following is a key lemma for our argument,
 that can be proven by applying the similar method to Lemma~3.5 of \cite{Kos29}.

\begin{lemma}\label{transl.}
Let $Z = (Z,\mathcal{U}_Z)$ be a non-degenerate space.
For each pseudometric $d \in \upm(Z,\mathcal{U}_Z)$ and each pair $(x,y) \in Z^2$ of distinct points, there is $\rho \in \upm(Z,\mathcal{U}_Z)$ such that
\begin{itemize}
 \item[($\ast$)] for every pair $(z,w) \in Z^2 \setminus \{(x,y),(y,x)\}$,
 $$\rho(x,y) \geq \rho(z,w) \text{ and } d(x,y) + \rho(x,y) > d(z,w) + \rho(z,w).$$
\end{itemize}
Additionally, in the case where $Z = (Z,d_Z)$ is a metric space and $d \in \lpm(Z,d_Z)$, $\rho$ can be chosen as a lipschitz metric,
 and hence $d + \rho \in \lpm(Z,d_Z)$.
\end{lemma}

\begin{proof}
Firstly, we will prove the former part.
Using Theorem~\ref{ext.}, we can easily find a pseudometric $e \in \upm(Z,\mathcal{U}_Z)$ such that $e(x,y) = \|e\| = 1$.
Let $d' = d + e$, $a = d'(x,y)$ and
 $$b = \min\bigg\{\sup_{z \in Z} d'(x,z),\sup_{z \in Z} d'(y,z)\bigg\}.$$
For each natural number $n \in \N$, take the closed subset
 $$W_n = Z \setminus (B_{d'}(x,a/2^n) \cup B_{d'}(y,a/2^n)) \cup \overline{B}_{d'}(x,a/2^{n + 1}) \cup \overline{B}_{d'}(y,a/2^{n + 1})$$
 and define a pseudometric $\rho_n$ on $W_n$ as follows:
\begin{enumerate}
\renewcommand{\theenumi}{\roman{enumi}}
 \item $\rho_n(z,w) = 4b$ if $z \in \overline{B}_{d'}(x,a/2^{n + 1})$ and $w \in \overline{B}_{d'}(y,a/2^{n + 1})$;
 \item $\rho_n(z,w) = 2b$ if $z \in Z \setminus (B_{d'}(x,a/2^n) \cup B_{d'}(y,a/2^n))$ and $w \in \overline{B}_{d'}(x,a/2^{n + 1}) \cup \overline{B}_{d'}(y,a/2^{n + 1})$;
 \item $\rho_n(z,w) = 0$ if $z, w \in Z \setminus (B_{d'}(x,a/2^n) \cup B_{d'}(y,a/2^n))$, if $z, w \in \overline{B}_{d'}(x,a/2^{n + 1})$,
 or if $z, w \in \overline{B}_{d'}(y,a/2^{n + 1})$.
\end{enumerate}
Observe the following:
\begin{enumerate}
\renewcommand{\theenumi}{\roman{enumi}}
 \item If $z \in \overline{B}_{d'}(x,a/2^{n + 1})$ and $w \in \overline{B}_{d'}(y,a/2^{n + 1})$,
 then
 $$\frac{\rho_n(z,w)}{d'(z,w)} \leq \frac{4b}{a(1 - 2^{-n})} \leq \frac{8b}{a}.$$
 \item If $z \in Z \setminus (B_{d'}(x,a/2^n) \cup B_{d'}(y,a/2^n))$ and $w \in \overline{B}_{d'}(x,a/2^{n + 1}) \cup \overline{B}_{d'}(y,a/2^{n + 1})$,
 then
 $$\frac{\rho_n(z,w)}{d'(z,w)} \leq \frac{2b}{a(2^{-n} - 2^{-n - 1})} = \frac{2^{n + 2}b}{a}.$$
\end{enumerate}
Since $d' \in \upm(Z,\mathcal{U}_Z)$,
 $\rho_n \in \upm(W_n,\mathcal{U}_Z|_{W_n})$.
By virtue of Theorem~\ref{ext.}, each $\rho_n$ can be extended over $Z^2$ such that $\rho_n$ is uniformly continuous with respect to $\mathcal{U}_Z$ and $\|\rho_n\| = 4b$.
Since $\upm(Z,\mathcal{U}_Z)$ is complete,
 the sum $\sum_{n = 1}^\infty \rho_n/2^n$ is converging in it.
Then
 $$\rho = e + \sum_{n = 1}^\infty \frac{\rho_n}{2^n} \in \upm(Z,\mathcal{U}_Z)$$
 is the desired pseudometric.
Indeed,
 $$\rho(x,y) = e(x,y) + \sum_{n = 1}^\infty \frac{\rho_n(x,y)}{2^n} = 1 + \sum_{n = 1}^\infty \frac{4b}{2^n} = 4b + 1 = \|\rho\|.$$
Putting $x = x_0$ and $y = x_1$, we have the following:
\begin{enumerate}
 \item $z = x_k$ and $w = x_{1 - k}$, $k \in \{0,1\}$.
 $$d(z,w) + \rho(z,w) = d(x,y) + \rho(x,y) = d(x,y) + 4b + 1.$$
 \item $z = x_k$, $w \in Z \setminus (\overline{B}_{d'}(x,a/2) \cup \overline{B}_{d'}(y,a/2))$, $k \in \{0,1\}$.
 \begin{align*}
  d(z,w) + \rho(z,w) &= d'(x_k,w) + \sum_{i = 1}^\infty \frac{\rho_i(x_k,w)}{2^i} \leq d'(x,y) + b + \sum_{i = 1}^\infty \frac{2b}{2^i}\\
  &= d(x,y) + 1 + b + \sum_{i = 1}^\infty \frac{2b}{2^i} = d(x,y) + 3b + 1.
 \end{align*}
 \item $z = x_k$, $w \in \overline{B}_{d'}(x_k,a/2^{n - 1}) \setminus B_{d'}(x_k,a/2^n)$, $k \in \{0,1\}$, $n \geq 2$.
 \begin{align*}
  d(z,w) + \rho(z,w) &= d'(x_k,w) + \sum_{i = 1}^\infty \frac{\rho_i(x_k,w)}{2^i} \leq \frac{a}{2^{n - 1}} + \frac{4b}{2^{n - 1}} + \sum_{i = n}^\infty \frac{2b}{2^i}\\
  &\leq \frac{4b}{2^{n + 1}} + \frac{4b}{2^{n - 1}} + \sum_{i = n + 1}^\infty \frac{4b}{2^i} < \sum_{i = 1}^\infty \frac{4b}{2^i} = 4b.
 \end{align*}
 \item $z = x_k$, $w \in \overline{B}_{d'}(x_{1 - k},a/2^{n - 1}) \setminus B_{d'}(x_{1 - k},a/2^n)$, $k \in \{0,1\}$, $n \geq 2$.
 \begin{align*}
  d(z,w) + \rho(z,w) &= d'(x_k,w) + \sum_{i = 1}^\infty \frac{\rho_i(x_k,w)}{2^i} \leq d'(x,y) + d'(x_{1 - k},w) + \sum_{i = 1}^\infty \frac{\rho_i(x_k,w)}{2^i}\\
  &\leq d(x,y) + 1 + \frac{a}{2^{n - 1}} + \sum_{i = 1}^{n - 1} \frac{4b}{2^n} + \sum_{i = n}^\infty \frac{2b}{2^i}\\
  &\leq d(x,y) + 1 + \frac{4b}{2^{n + 1}} + \sum_{i = 1}^{n - 1} \frac{4b}{2^i} + \sum_{i = n + 1}^\infty \frac{4b}{2^i}\\
  &< d(x,y) + 1 + \sum_{i = 1}^\infty \frac{4b}{2^i} = d(x,y) + 4b + 1.
 \end{align*}
 \item $z, w \in Z \setminus (\overline{B}_{d'}(x,a/2) \cup \overline{B}_{d'}(y,a/2))$.
  $$d(z,w) + \rho(z,w) = d'(z,w) + \sum_{i = 1}^\infty \frac{\rho_i(z,w)}{2^i} \leq d'(x,y) + 2b = d(x,y) + 2b + 1.$$
 \item $z, w \in \overline{B}_{d'}(x_k,a/2^{n - 1}) \setminus B_{d'}(x_k,a/2^n)$, $k \in \{0,1\}$, $n \geq 2$.
 \begin{align*}
  d(z,w) + \rho(z,w) &= d'(z,w) + \sum_{i = 1}^\infty \frac{\rho_i(z,w)}{2^i} \leq d'(x_k,z) + d'(x_k,w) + \sum_{i = 1}^\infty \frac{\rho_i(z,w)}{2^i}\\
  &\leq \frac{a}{2^{n - 2}} + \frac{4b}{2^{n - 1}} \leq \frac{b}{2^{n - 2}} + \frac{4b}{2^{n - 1}} \leq 3b.
 \end{align*}
 \item $z \in \overline{B}_{d'}(x_k,a/2^{n - 1}) \setminus B_{d'}(x_k,a/2^n)$, $w \in \overline{B}_{d'}(x_{1 - k},a/2^{n - 1}) \setminus B_{d'}(x_{1 - k},a/2^n)$, $k \in \{0,1\}$, $n \geq 2$.
 \begin{align*}
  d(z,w) + \rho(z,w) &= d'(z,w) + \sum_{i = 1}^\infty \frac{\rho_i(z,w)}{2^i}\\
  &\leq d'(x_k,z) + d'(x,y) + d'(x_{1 - k},w) + \sum_{i = 1}^\infty \frac{\rho_i(z,w)}{2^i}\\
  &\leq d(x,y) + 1 + \frac{a}{2^{n - 2}} + \sum_{i = 1}^{n - 1} \frac{4b}{2^i} \leq d(x,y) + 1 + \frac{4b}{2^n} + \sum_{i = 1}^{n - 1} \frac{4b}{2^i}\\
  &< d(x,y) + 1 + \sum_{i = 1}^\infty \frac{4b}{2^i} = d(x,y) + 4b + 1.
 \end{align*}
 \item $z \in Z \setminus (\overline{B}_{d'}(x,a/2) \cup \overline{B}_{d'}(y,a/2))$ and $w \in (\overline{B}_{d'}(x,a/2) \cup \overline{B}_{d'}(y,a/2)) \setminus (B_{d'}(x,a/4) \cup B_{d'}(y,a/4))$.
 \begin{align*}
  d(z,w) + \rho(z,w) &= d'(z,w) + \sum_{i = 1}^\infty \frac{\rho_i(z,w)}{2^i} \leq d'(x,y) + b + \frac{a}{2} + \frac{4b}{2}\\
  &\leq d(x,y) + 1 + b + \frac{b}{2} + \frac{4b}{2} = d(x,y) + \frac{7b}{2} + 1.
 \end{align*}
 \item $z \in \overline{B}_{d'}(x_k,a/2^{n - 1}) \setminus B_{d'}(x_k,a/2^n)$, $w \in \overline{B}_{d'}(x_k,a/2^n) \setminus B_{d'}(x_k,a/2^{n + 1})$, $k \in \{0,1\}$, $n \geq 2$.
 \begin{align*}
  d(z,w) + \rho(z,w) &= d'(z,w) + \sum_{i = 1}^\infty \frac{\rho_i(z,w)}{2^i} \leq d'(x_k,z) + d'(x_k,w) + \sum_{i = 1}^\infty \frac{\rho_i(z,w)}{2^i}\\
  &\leq \frac{a}{2^{n - 1}} + \frac{a}{2^n} + \sum_{i = n - 1}^n \frac{4b}{2^i} \leq \frac{b}{2^{n - 1}} + \frac{b}{2^n} + \frac{4b}{2^{n - 1}} + \frac{4b}{2^n} \leq \frac{15b}{4}.
 \end{align*}
 \item $z \in \overline{B}_{d'}(x_k,a/2^{n - 1}) \setminus B_{d'}(x_k,a/2^n)$, $w \in \overline{B}_{d'}(x_{1 - k},a/2^n) \setminus B_{d'}(x_{1 - k},a/2^{n + 1})$, $k \in \{0,1\}$, $n \geq 2$.
 \begin{align*}
  d(z,w) + \rho(z,w) &= d'(z,w) + \sum_{i = 1}^\infty \frac{\rho_i(z,w)}{2^i}\\
  &\leq d'(x_k,z) + d'(x,y) + d'(x_{1 - k},w) + \sum_{i = 1}^\infty \frac{\rho_i(z,w)}{2^i}\\
  &\leq d(x,y) + 1 + \frac{a}{2^{n - 1}} + \frac{a}{2^n} + \sum_{i = 1}^n \frac{4b}{2^i}\\
  &\leq d(x,y) + 1 + \frac{4b}{2^{n + 1}} + \frac{4b}{2^{n + 2}} + \sum_{i = 1}^n \frac{4b}{2^i}\\
  &< d(x,y) + 1 + \sum_{i = 1}^\infty \frac{4b}{2^i} = d(x,y) + 4b + 1.
 \end{align*}
 \item $z \in Z \setminus (\overline{B}_{d'}(x,a/2) \cup \overline{B}_{d'}(y,a/2))$ and $w \in (\overline{B}_{d'}(x,a/2^{m - 1}) \cup \overline{B}_{d'}(y,a/2^{m - 1})) \setminus (B_{d'}(x,a/2^m) \cup B_{d'}(y,a/2^m))$, $m \geq 3$.
 \begin{align*}
  d(z,w) + \rho(z,w) &= d'(z,w) + \sum_{i = 1}^\infty \frac{\rho_i(z,w)}{2^i} \leq d'(x,y) + b + \frac{a}{2^{m - 1}} + \sum_{i = 1}^{m - 2} \frac{2b}{2^i} + \frac{4b}{2^{m - 1}}\\
  &\leq d(x,y) + 1 + b + \frac{4b}{2^{m + 1}} + \sum_{i = 2}^{m - 1} \frac{4b}{2^i} + \frac{4b}{2^{m - 1}}\\
  &= d(x,y) + 1 + \frac{4b}{2} - \frac{4b}{2^2} + \frac{4b}{2^{m + 1}} + \sum_{i = 2}^{m - 1} \frac{4b}{2^i} + \frac{4b}{2^{m - 1}}\\
  &< d(x,y) + 1 + \sum_{i = 1}^\infty \frac{4b}{2^i} = d(x,y) + 4b + 1.
 \end{align*}
 \item $z \in \overline{B}_{d'}(x_k,a/2^{n - 1}) \setminus B_{d'}(x_k,a/2^n)$, $w \in \overline{B}_{d'}(x_k,a/2^{m - 1}) \setminus B_{d'}(x_k,a/2^m)$, $k \in \{0,1\}$, $n \geq 2$, $m \geq n + 2$.
 \begin{align*}
  d(z,w) + \rho(z,w) &= d'(z,w) + \sum_{i = 1}^\infty \frac{\rho_i(z,w)}{2^i} \leq d'(x_k,z) + d'(x_k,w) + \sum_{i = 1}^\infty \frac{\rho_i(z,w)}{2^i}\\
  &\leq \frac{a}{2^{n - 1}} + \frac{a}{2^{m - 1}} + \frac{4b}{2^{n - 1}} + \sum_{i = n}^{m - 2} \frac{2b}{2^i} + \frac{4b}{2^{m - 1}}\\
  &\leq \frac{4b}{2^{n + 1}} + \frac{4b}{2^{m + 1}} + \frac{4b}{2^{n - 1}} + \sum_{i = n + 1}^{m - 1} \frac{4b}{2^i}+ \frac{4b}{2^{m - 1}}\\
  &= \frac{4b}{2^n} - \frac{4b}{2^{n + 1}} + \frac{4b}{2^{m + 1}} + \frac{4b}{2^{n - 1}} + \sum_{i = n + 1}^{m - 1} \frac{4b}{2^i} + \frac{4b}{2^{m - 1}}\\
  &< \sum_{i = 1}^\infty \frac{4b}{2^i} = 4b.
 \end{align*}
 \item $z \in \overline{B}_{d'}(x_k,a/2^{n - 1}) \setminus B_{d'}(x_k,a/2^n)$, $w \in \overline{B}_{d'}(x_{1 - k},a/2^{m - 1}) \setminus B_{d'}(x_{1 - k},a/2^m)$, $k \in \{0,1\}$, $n \geq 2$, $m \geq n + 2$.
 \begin{align*}
  d(z,w) + \rho(z,w) &= d'(z,w) + \sum_{i = 1}^\infty \frac{\rho_i(z,w)}{2^i}\\
  &\leq d'(x_k,z) + d'(x,y) + d'(x_{1 - k},w) + \sum_{i = 1}^\infty \frac{\rho_i(z,w)}{2^i}\\
  &\leq d(x,y) + 1 + \frac{a}{2^{n - 1}} + \frac{a}{2^{m - 1}} + \sum_{i = 1}^{n - 1} \frac{4b}{2^i} + \sum_{i = n}^{m - 2} \frac{2b}{2^i} + \frac{4b}{2^{m - 1}}\\
  &\leq d(x,y) + 1 + \frac{4b}{2^{n + 1}} + \frac{4b}{2^{m + 1}} + \sum_{i = 1}^{n - 1} \frac{4b}{2^i} + \sum_{i = n + 1}^{m - 1} \frac{4b}{2^i} + \frac{4b}{2^{m - 1}}\\
  &= d(x,y) + 1 + \frac{4b}{2^n} - \frac{4b}{2^{n + 1}} + \frac{4b}{2^{m + 1}} + \sum_{i = 1}^{n - 1} \frac{4b}{2^i} + \sum_{i = n + 1}^{m - 1} \frac{4b}{2^i} + \frac{4b}{2^{m - 1}}\\
  &< d(x,y) + 1 + \sum_{i = 1}^\infty \frac{4b}{2^i} = d(x,y) + 4b + 1.
 \end{align*}
\end{enumerate}
Consequently, for every pair $(z,w) \in Z^2 \setminus \{(x,y),(y,x)\}$,
 $$d(x,y) + \rho(x,y) = d(x,y) + 4b + 1 = \|d + \rho\| > d(z,w) + \rho(z,w).$$

In the latter case, replace the above metric $e$ with
 $$e = \min\bigg\{\frac{d_Z}{d_Z(x,y)},1\bigg\},$$
 and define the metric $d'$ on $Z$, the positive numbers $a, b > 0$, the closed subsets $W_n \subset Z$, and the metrics $\rho_n$ on $W_n$ by the same way.
Then $e \in \lpm(Z)$ with $e(x,y) = \|e\| = 1$ and $\lipc(e) = 1/d_Z(x,y)$.
For any $z, w \in W_n$ with $z \neq w$, we will check the following:
\begin{enumerate}
\renewcommand{\theenumi}{\roman{enumi}}
 \item When $z \in \overline{B}_{d'}(x,a/2^{n + 1})$ and $w \in \overline{B}_{d'}(y,a/2^{n + 1})$,
 $$\frac{\rho_n(z,w)}{d_Z(z,w)} = \frac{\rho_n(z,w)}{d'(z,w)} \frac{d'(z,w)}{d_Z(z,w)} \leq \frac{4b}{a(1 - 2^{-n})}(\lipc(d) + \lipc(e)) \leq \frac{8b(\lipc(d) + \lipc(e))}{a}.$$
 \item When $z \in Z \setminus (B_{d'}(x,a/2^n) \cup B_{d'}(y,a/2^n))$ and $w \in \overline{B}_{d'}(x,a/2^{n + 1}) \cup \overline{B}_{d'}(y,a/2^{n + 1})$,
 $$\frac{\rho_n(z,w)}{d_Z(z,w)} = \frac{\rho_n(z,w)}{d'(z,w)} \frac{d'(z,w)}{d_Z(z,w)} \leq \frac{2b}{a(2^{-n} - 2^{-n - 1})}(\lipc(d) + \lipc(e)) = \frac{2^{n + 2}b(\lipc(d) + \lipc(e))}{a}.$$
 \item When $z, w \in Z \setminus (B_{d'}(x,a/2^n) \cup B_{d'}(y,a/2^n))$,
 when $z, w \in \overline{B}_{d'}(x,a/2^{n + 1})$,
 or when $z, w \in \overline{B}_{d'}(y,a/2^{n + 1})$,
 $$\frac{\rho_n(z,w)}{d_Z(z,w)} = 0.$$
\end{enumerate}
Applying Theorem~\ref{ext.}, we can extend $\rho_n$ over $Z^2$ such that $\|\rho_n\| = 4b$ and
 $$\lipc(\rho_n) \leq \max\bigg\{\frac{8b(\lipc(d) + \lipc(e))}{a},\frac{2^{n + 2}b(\lipc(d) + \lipc(e))}{a}\bigg\} \leq \frac{2^{n + 2}b(\lipc(d) + \lipc(e))}{a}.$$

As has been observed,
 the metric
 $$\rho = e + \sum_{n = 1}^\infty \frac{\rho_n}{2^n} \in \upm(Z,d_Z)$$
 fulfills the property ($\ast$).
It remains to show that $\rho$ is lipschitz.
For all distinct points $z, w \in Z$, verify the following:
\begin{enumerate}
 \item $z = x_k$, $w = x_{1 - k}$, $k \in \{0,1\}$.
 \begin{align*}
  \frac{\sum_{i = 1}^\infty \rho_i(z,w)/2^i}{d_Z(z,w)} &= \frac{\sum_{i = 1}^\infty \rho_i(x,y)/2^i}{d_Z(x,y)} = \frac{\sum_{i = 1}^\infty 4b/2^i}{d_Z(x,y)} = \frac{4b}{d'(x,y)} \frac{d'(x,y)}{d_Z(x,y)}\\
  &\leq \frac{4b(\lipc(d) + \lipc(e))}{a}.
 \end{align*}
 \item $z = x_k$, $w \in Z \setminus (\overline{B}_{d'}(x,a/2) \cup \overline{B}_{d'}(y,a/2))$, $k \in \{0,1\}$.
 \begin{align*}
  \frac{\sum_{i = 1}^\infty \rho_i(z,w)/2^i}{d_Z(z,w)} &= \frac{\sum_{i = 1}^\infty \rho_i(x_k,w)/2^i}{d_Z(x_k,w)} = \frac{\sum_{i = 1}^\infty 2b/2^i}{d_Z(x_k,w)} = \frac{2b}{d'(x_k,w)} \frac{d'(x_k,w)}{d_Z(x_k,w)}\\
  &\leq \frac{2b(\lipc(d) + \lipc(e))}{a/2} = \frac{4b(\lipc(d) + \lipc(e))}{a}.
 \end{align*}
 \item $z = x_k$, $w \in \overline{B}_{d'}(x_k,a/2^{n - 1}) \setminus B_{d'}(x_k,a/2^n)$, $k \in \{0,1\}$, $n \geq 2$.
 \begin{align*}
  \frac{\sum_{i = 1}^\infty \rho_i(z,w)/2^i}{d_Z(z,w)} &= \frac{\sum_{i = 1}^\infty \rho_i(x_k,w)/2^i}{d_Z(x_k,w)} = \frac{\rho_{n - 1}(x_k,w)/2^{n - 1} + \sum_{i = n}^\infty 2b/2^i}{d_Z(x_k,w)}\\
  &= \frac{\rho_{n - 1}(x_k,w)/2^{n - 1}}{d_Z(x_k,w)} + \frac{2b/2^{n - 1}}{d'(x_k,w)} \frac{d'(x_k,w)}{d_Z(x_k,w)}\\
  &\leq \frac{2^{n + 1}b(\lipc(d) + \lipc(e))/2^{n - 1}}{a} + \frac{2b(\lipc(d) + \lipc(e))/2^{n - 1}}{a/2^n}\\
  &= \frac{8b(\lipc(d) + \lipc(e))}{a}.
 \end{align*}
 \item $z = x_k$, $w \in \overline{B}_{d'}(x_{1 - k},a/2^{n - 1}) \setminus B_{d'}(x_{1 - k},a/2^n)$, $k \in \{0,1\}$, $n \geq 2$.
 \begin{align*}
  \frac{\sum_{i = 1}^\infty \rho_i(z,w)/2^i}{d_Z(z,w)} &= \frac{\sum_{i = 1}^\infty \rho_i(x_k,w)/2^i}{d_Z(x_k,w)} = \frac{\sum_{i = 1}^{n - 2} 4b/2^i + \rho_{n - 1}(x_k,w)/2^{n - 1} + \sum_{i = n}^\infty 2b/2^i}{d_Z(x_k,w)}\\
  &= \frac{4b(1 - 1/2^{n - 2})}{d'(x_k,w)} \frac{d'(x_k,w)}{d_Z(x_k,w)} + \frac{\rho_{n - 1}(x_k,w)/2^{n - 1}}{d_Z(x_k,w)} + \frac{2b/2^{n - 1}}{d'(x_k,w)} \frac{d'(x_k,w)}{d_Z(x_k,w)}\\
  &\leq \frac{4b(\lipc(d) + \lipc(e))(1 - 1/2^{n - 2})}{a(1 - 1/2^{n - 1})} + \frac{2^{n + 1}b(\lipc(d) + \lipc(e))/2^{n - 1}}{a}\\
  &\ \ \ \ \ \ \ \ \ \ \ \ \ \ \ \ \ \ \ \ \ \ \ \ \ \ \ \ \ \ \ \ \ \ \ \ \ \ \ \ \ \ \ \ \ \ \ \ + \frac{2b(\lipc(d) + \lipc(e))/2^{n - 1}}{a(1 - 1/2^{n - 1})}\\
  &\leq \frac{12b(\lipc(d) + \lipc(e))}{a}.
 \end{align*}
 \item $z, w \in Z \setminus (\overline{B}_{d'}(x,a/2) \cup \overline{B}_{d'}(y,a/2))$.
 $$\frac{\sum_{i = 1}^\infty \rho_i(z,w)/2^i}{d_Z(z,w)} = 0.$$
 \item $z, w \in \overline{B}_{d'}(x_k,a/2^{n - 1}) \setminus B_{d'}(x_k,a/2^n)$, $k \in \{0,1\}$, $n \geq 2$.
 $$\frac{\sum_{i = 1}^\infty \rho_i(z,w)/2^i}{d_Z(z,w)} = \frac{\rho_{n - 1}(z,w)/2^{n - 1}}{d_Z(z,w)} \leq \frac{2^{n + 1}b(\lipc(d) + \lipc(e))/2^{n - 1}}{a} = \frac{4b(\lipc(d) + \lipc(e))}{a}.$$
 \item $z \in \overline{B}_{d'}(x_k,a/2^{n - 1}) \setminus B_{d'}(x_k,a/2^n)$, $w \in \overline{B}_{d'}(x_{1 - k},a/2^{n - 1}) \setminus B_{d'}(x_{1 - k},a/2^n)$, $k \in \{0,1\}$, $n \geq 2$.
 \begin{align*}
  \frac{\sum_{i = 1}^\infty \rho_i(z,w)/2^i}{d_Z(z,w)} &= \frac{\sum_{i = 1}^{n - 2} 4b/2^i + \rho_{n - 1}(z,w)/2^{n - 1}}{d_Z(z,w)}\\
  &= \frac{4b(1 - 1/2^{n - 2})}{d'(z,w)} \frac{d'(z,w)}{d_Z(z,w)} + \frac{\rho_{n - 1}(z,w)/2^{n - 1}}{d_Z(z,w)}\\
  &\leq \frac{4b(\lipc(d) + \lipc(e))(1 - 1/2^{n - 2})}{a(1 - 1/2^{n - 2})} + \frac{2^{n + 1}b(\lipc(d) + \lipc(e))/2^{n - 1}}{a}\\
  &\leq \frac{8b(\lipc(d) + \lipc(e))}{a}.
 \end{align*}
 \item $z \in Z \setminus (\overline{B}_{d'}(x,a/2) \cup \overline{B}_{d'}(y,a/2))$, $w \in (\overline{B}_{d'}(x,a/2) \cup \overline{B}_{d'}(y,a/2)) \setminus (B_{d'}(x,a/4) \cup B_{d'}(y,a/4))$.
 $$\frac{\sum_{i = 1}^\infty \rho_i(z,w)/2^i}{d_Z(z,w)} = \frac{\rho_1(z,w)/2}{d_Z(z,w)} \leq \frac{2^3b(\lipc(d) + \lipc(e))/2}{a} = \frac{4b(\lipc(d) + \lipc(e))}{a}.$$
 \item $z \in \overline{B}_{d'}(x_k,a/2^{n - 1}) \setminus B_{d'}(x_k,a/2^n)$, $w \in \overline{B}_{d'}(x_k,a/2^n) \setminus B_{d'}(x_k,a/2^{n + 1})$, $k \in \{0,1\}$, $n \geq 2$.
 \begin{align*}
  \frac{\sum_{i = 1}^\infty \rho_i(z,w)/2^i}{d_Z(z,w)} &= \frac{\rho_{n - 1}(z,w)/2^{n - 1} + \rho_n(z,w)/2^n}{d_Z(z,w)}\\
  &\leq \frac{2^{n + 1}b(\lipc(d) + \lipc(e))/2^{n - 1}}{a} + \frac{2^{n + 2}b(\lipc(d) + \lipc(e))/2^n}{a}\\
  &= \frac{8b(\lipc(d) + \lipc(e))}{a}.
 \end{align*}
 \item $z \in \overline{B}_{d'}(x_k,a/2^{n - 1}) \setminus B_{d'}(x_k,a/2^n)$, $w \in \overline{B}_{d'}(x_{1 - k},a/2^n) \setminus B_{d'}(x_{1 - k},a/2^{n + 1})$, $k \in \{0,1\}$, $n \geq 2$.
 \begin{align*}
  \frac{\sum_{i = 1}^\infty \rho_i(z,w)/2^i}{d_Z(z,w)} &= \frac{\sum_{i = 1}^{n - 2} 4b/2^i + \rho_{n - 1}(z,w)/2^{n - 1} + \rho_n(z,w)/2^n}{d_Z(z,w)}\\
  &= \frac{4b(1 - 1/2^{n - 2})}{d'(z,w)} \frac{d'(z,w)}{d_Z(z,w)} + \frac{\rho_{n - 1}(z,w)/2^{n - 1}}{d_Z(z,w)} + \frac{\rho_n(z,w)/2^n}{d_Z(z,w)}\\
  &\leq \frac{4b(\lipc(d) + \lipc(e))(1 - 1/2^{n - 2})}{a(1 - 1/2^{n - 1} - 1/2^n)} + \frac{2^{n + 1}b(\lipc(d) + \lipc(e))/2^{n - 1}}{a}\\
  &\ \ \ \ \ \ \ \ \ \ \ \ \ \ \ \ \ \ \ \ \ \ \ \ \ \ \ \ \ \ \ \ \ \ \ \ \ \ \ \ \ \ \ \ \ \ \ \ + \frac{2^{n + 2}b(\lipc(d) + \lipc(e))/2^n}{a}\\
  &\leq \frac{24b(\lipc(d) + \lipc(e))}{a}.
 \end{align*}
 \item $z \in Z \setminus (\overline{B}_{d'}(x,a/2) \cup \overline{B}_{d'}(y,a/2))$, $w \in (\overline{B}_{d'}(x,a/2^{m - 1}) \cup \overline{B}_{d'}(y,a/2^{m - 1})) \setminus (B_{d'}(x,a/2^m) \cup B_{d'}(y,a/2^m))$, $m \geq 3$.
 \begin{align*}
  \frac{\sum_{i = 1}^\infty \rho_i(z,w)/2^i}{d_Z(z,w)} &= \frac{\sum_{i = 1}^{m - 2} 2b/2^i + \rho_{m - 1}(z,w)/2^{m - 1}}{d_Z(z,w)}\\
  &\leq \frac{2b(1 - 1/2^{m - 2})}{d'(z,w)} \frac{d'(z,w)}{d_Z(z,w)} + \frac{\rho_{m - 1}(z,w)/2^{m - 1}}{d_Z(z,w)}\\
  &= \frac{2b(\lipc(d) + \lipc(e))(1 - 1/2^{m - 2})}{a(1 - 1/2 - 1/2^{m - 1})} + \frac{2^{m + 1}b(\lipc(d) + \lipc(e))/2^{m - 1}}{a}\\
  &\leq \frac{12b(\lipc(d) + \lipc(e))}{a}.
 \end{align*}
 \item $z \in \overline{B}_{d'}(x_k,a/2^{n - 1}) \setminus B_{d'}(x_k,a/2^n)$, $w \in \overline{B}_{d'}(x_k,a/2^{m - 1}) \setminus B_{d'}(x_k,a/2^m)$, $k \in \{0,1\}$, $n \geq 2$, $m \geq n + 2$.
 \begin{align*}
  \frac{\sum_{i = 1}^\infty \rho_i(z,w)/2^i}{d_Z(z,w)} &= \frac{\rho_{n - 1}(z,w)/2^{n - 1} + \sum_{i = n}^{m - 2} 2b/2^i + \rho_{m - 1}(z,w)/2^{m - 1}}{d_Z(z,w)}\\
  &= \frac{\rho_{n - 1}(z,w)/2^{n - 1}}{d_Z(z,w)} + \frac{2b(1/2^{n - 1} - 1/2^{m - 2})}{d'(z,w)} \frac{d'(z,w)}{d_Z(z,w)} + \frac{\rho_{m - 1}(z,w)/2^{m - 1}}{d_Z(z,w)}\\
  &\leq \frac{2^{n + 1}b(\lipc(d) + \lipc(e))/2^{n - 1}}{a} + \frac{2b(\lipc(d) + \lipc(e))(1/2^{n - 1} - 1/2^{m - 2})}{a(1/2^n - 1/2^{m - 1})}\\
  &\ \ \ \ \ \ \ \ \ \ \ \ \ \ \ \ \ \ \ \ \ \ \ \ \ \ \ \ \ \ \ \ \ \ \ \ \ \ \ \ \ \ \ \ \ \ \ \ + \frac{2^{m + 1}b(\lipc(d) + \lipc(e))/2^{m - 1}}{a}\\
  &\leq \frac{12b(\lipc(d) + \lipc(e))}{a}.
 \end{align*}
 \item $z \in \overline{B}_{d'}(x_k,a/2^{n - 1}) \setminus B_{d'}(x_k,a/2^n)$, $w \in \overline{B}_{d'}(x_{1 - k},a/2^{m - 1}) \setminus B_{d'}(x_{1 - k},a/2^m)$, $k \in \{0,1\}$, $n \geq 2$, $m \geq n + 2$.
 \begin{align*}
  \frac{\sum_{i = 1}^\infty \rho_i(z,w)/2^i}{d_Z(z,w)} &= \frac{\sum_{i = 1}^{n - 2} 4b/2^i + \rho_{n - 1}(z,w)/2^{n - 1} + \sum_{i = n}^{m - 2} 2b/2^i + \rho_{m - 1}(z,w)/2^{m - 1}}{d_Z(z,w)}\\
  &= \frac{4b(1 - 1/2^{n - 2})}{d'(z,w)} \frac{d'(z,w)}{d_Z(z,w)} + \frac{\rho_{n - 1}(z,w)/2^{n - 1}}{d_Z(z,w)}\\
  &\ \ \ \ \ \ \ \ + \frac{2b(1/2^{n - 1} - 1/2^{m - 2})}{d'(z,w)} \frac{d'(z,w)}{d_Z(z,w)} + \frac{\rho_{m - 1}(z,w)/2^{m - 1}}{d_Z(z,w)}\\
  &\leq \frac{4b(\lipc(d) + \lipc(e))(1 - 1/2^{n - 2})}{a(1 - 1/2^{n - 1} - 1/2^{m - 1})} + \frac{2^{n + 1}b(\lipc(d) + \lipc(e))/2^{n - 1}}{a}\\
  &\ \ \ \ \ \ \ \ + \frac{2b(\lipc(d) + \lipc(e))(1/2^{n - 1} - 1/2^{m - 2})}{a(1 - 1/2^{n - 1} - 1/2^{m - 1})} + \frac{2^{m + 1}b(\lipc(d) + \lipc(e))/2^{m - 1}}{a}\\
  &\leq \frac{12b(\lipc(d) + \lipc(e))}{a}.
 \end{align*}
\end{enumerate}
It follows that
 $$\lipc(\rho) \leq \lipc(e) + \lipc\Bigg(\sum_{n = 1}^\infty \frac{\rho_n}{2^n}\Bigg) \leq \lipc(e) + \frac{24b(\lipc(d) + \lipc(e))}{a} < \infty.$$
We complete the proof.
\end{proof}

Given a space $Z$, denote the hyperspace consisting of non-empty finite subsets of $Z$ with cardinality $\leq 2$ by $\fin_2(Z)$,
 and denote the subspace consisting of doubletons by $\doubl(Z)$,
 which are equipped with the Vietoris topology.
For each $\{x,y\} \in \fin_2(Z)$ and each $d \in \metr(Z)$, we define the following subsets in $\metr(Z)$ and $\fin_2(Z)$:
 $$\mathcal{P}(Z,\{x,y\}) = \{\rho \in \metr(Z) \mid \rho(x,y) = \|\rho\|\},$$
 $$\mathcal{F}(Z,d) = \{\{z,w\} \in \fin_2(Z) \mid d(z,w) = \|d\|\}.$$
In the rest of this section, we will assume that
 $$T({\rm P}(X) \cap {\rm M}(X)) = {\rm P}(Y) \cap {\rm M}(Y).$$

\begin{lemma}\label{f.i.p.}
For every $\{x,y\} \in \fin_2(X)$ and every $\{z,w\} \in \fin_2(Y)$, the both of families
 $$\{\mathcal{F}(Y,T(d)) \mid d \in \mathcal{P}(X,\{x,y\}) \cap {\rm P}(X)\} \text{ and } \{\mathcal{F}(X,T^{-1}(\rho)) \mid \rho \in \mathcal{P}(Y,\{z,w\}) \cap {\rm P}(Y)\}$$
 have the finite intersection property.
\end{lemma}

\begin{proof}
It is sufficient to show that for any $n \in \N$ and any $d_i \in \mathcal{P}(X,\{x,y\}) \cap {\rm P}(X)$, $1 \leq i \leq n$,
 $$\bigcap_{i = 1}^n \mathcal{F}(Y,T(d_i)) \neq \emptyset.$$
In the case where $\metr(X) = \lpm_k(X,d_X)$ and $\metr(Y) = \lpm_k(Y,d_Y)$, let $d = \sum_{i = 1}^n d_i/n$,
 so due to the same argument as Lemma~3.3 of \cite{Kos29}, $d \in \mathcal{P}(X,\{x,y\})$,
 and moreover by Proposition~\ref{norm} and the assumption, there is $\{z,w\} \in \fin_2(Y)$ such that
 $$T(d)(z,w) = \|T(d)\| = \|d\| = d(x,y).$$
We will prove that $\{z,w\} \in \mathcal{F}(Y,T(d_i))$ for each $i \in \{1, \ldots, n\}$.
Put $\rho_i = d - d_i/n$,
 so it also satisfies that $\rho_i(x,y) = \|\rho_i\|$ like the above.
Since $T$ is isometric,
 $$T(d)(z,w) - T\bigg(\frac{d_i}{n}\bigg)(z,w) \leq \bigg\|T(d) - T\bigg(\frac{d_i}{n}\bigg)\bigg\| = \bigg\|d - \frac{d_i}{n}\bigg\| = \|\rho_i\| = \rho_i(x,y).$$
Using Proposition~\ref{norm}, we get that
\begin{align*}
 \bigg\|T\bigg(\frac{d_i}{n}\bigg)\bigg\| &= \bigg\|\frac{d_i}{n}\bigg\| = \frac{\|d_i\|}{n} = \frac{d_i(x,y)}{n} = d(x,y) - \rho_i(x,y) = T(d)(z,w) - \rho_i(x,y)\\
 &\leq T\bigg(\frac{d_i}{n}\bigg)(z,w) \leq \bigg\|T\bigg(\frac{d_i}{n}\bigg)\bigg\|.
\end{align*}
Due to Proposition~\ref{scalar}, $T(d_i/n) = T(d_i)/n$,
 and hence $T(d_i)(z,w) = \|T(d_i)\|$.
Consequently, $\{z,w\} \in \mathcal{F}(Y,T(d_i))$.
\end{proof}

There exists a bijection between $\doubl(Y)$ and $\doubl(X)$,
 that is compatible with the isometry $T$.

\begin{proposition}\label{eq.}
There is a bijection $\Phi : \doubl(Y) \to \doubl(X)$ such that for all $d \in \metr(X)$ and all $\{x,y\} \in \doubl(Y)$,
 $$T(d)(x,y) = d(\Phi(\{x,y\})),$$
 where if $\Phi(\{x,y\}) = \{z,w\} \in \doubl(X)$,
 then $d(\Phi\{x,y\}) = d(z,w)$.
\end{proposition}

\begin{proof}
According to the same argument as \cite[Lemmas~3.6 and 3.7, and Proposition~3.8]{Kos29}, applying Theorem~\ref{ext.}, Lemmas~\ref{f.i.p.} and \ref{transl.}, for each $\{x,y\} \in \doubl(Y)$, we can define $\Phi$ by
 $$\{\Phi(\{x,y\})\} = \bigcap_{\rho \in \mathcal{P}(Y,\{x,y\}) \cap {\rm P}(Y)} \mathcal{F}(X,T^{-1}(\rho)),$$
 which is the desired bijection.
Then for all $\{z,w\} \in \doubl(X)$ and all $\{x,y\} \in \doubl(Y)$,
 $$T(\mathcal{P}(X,\{z,w\}) \cap {\rm P}(X)) = \mathcal{P}(Y,\Phi^{-1}(\{z,w\})) \cap {\rm P}(Y),$$
 $$T^{-1}(\mathcal{P}(Y,\{x,y\}) \cap {\rm P}(Y)) = \mathcal{P}(X,\Phi(\{x,y\})) \cap {\rm P}(X).$$

We shall show the equality in the case that $\metr(X) = \lpm_k(X,d_X)$ and $\metr(Y) = \lpm_k(Y,d_Y)$.
To prove that $T(d)(x,y) \leq d(\Phi(\{x,y\}))$, choose $\rho \in \lpm(Y,d_Y)$ as in Lemma~\ref{transl.} so that $\rho(x,y) = \|\rho\|$ and
 $$\frac{1}{\lambda}(T(d) + \rho) \in \mathcal{P}(Y,\{x,y\}) \cap {\rm P}(Y)$$
 for some $\lambda \geq 1$.
Since
 $$T^{-1}(\mathcal{P}(Y,\{x,y\}) \cap {\rm P}(Y)) = \mathcal{P}(X,\Phi(\{x,y\})) \cap {\rm P}(X),$$
 we get that
 $$T^{-1}\bigg(\frac{1}{\lambda}(T(d) + \rho)\bigg) \in \mathcal{P}(X,\Phi(\{x,y\})).$$
Observe that by Proposition~\ref{norm},
\begin{align*}
 \frac{1}{\lambda}(T(d)(x,y) + \rho(x,y)) &= \frac{1}{\lambda}(T(d) + \rho)(x,y) = \bigg\|\frac{1}{\lambda}(T(d) + \rho)\bigg\| = \bigg\|T^{-1}\bigg(\frac{1}{\lambda}(T(d) + \rho)\bigg)\bigg\|\\
 &= T^{-1}\bigg(\frac{1}{\lambda}(T(d) + \rho)\bigg)(\Phi(\{x,y\})).
\end{align*}
Since $T$ is an isometry and scalar-preserving due to Proposition~\ref{scalar},
 \begin{align*}
  &\ \ \ \ \frac{1}{\lambda}(T(d)(x,y) + \rho(x,y)) - \frac{1}{\lambda} d(\Phi(\{x,y\}))\\
  &= T^{-1}\bigg(\frac{1}{\lambda}(T(d) + \rho)\bigg)(\Phi(\{x,y\})) - \frac{1}{\lambda} d(\Phi(\{x,y\}))\\
  &\leq \bigg\|T^{-1}\bigg(\frac{1}{\lambda}(T(d) + \rho)\bigg) - \frac{1}{\lambda} d\bigg\| = \bigg\|T^{-1}\bigg(\frac{1}{\lambda}(T(d) + \rho)\bigg) - T^{-1}\bigg(T\bigg(\frac{1}{\lambda} d\bigg)\bigg)\bigg\|\\
  &= \bigg\|\frac{1}{\lambda}(T(d) + \rho) - T\bigg(\frac{1}{\lambda} d\bigg)\bigg\|= \bigg\|\frac{1}{\lambda}(T(d) + \rho) - \frac{1}{\lambda} T(d)\bigg\| = \frac{1}{\lambda}\|\rho\| = \frac{1}{\lambda}\rho(x,y).
 \end{align*}
Hence $T(d)(x,y) \leq d(\Phi(\{x,y\}))$.
The converse inequality follows from the same way as the above.
Consequently, $T(d)(x,y) = d(\Phi(\{x,y\}))$.
\end{proof}

We are now in a position to prove that the isometry $T$ is a composition operator by some homeomorphism from $Y$ to $X$.

\begin{proposition}\label{corresp.}
There exists a homeomorphism $\phi : Y \to X$ such that
 $$T(d)(x,y) = d(\phi(x),\phi(y))$$
 holds for any $d \in \metr(X)$ for any $\{x,y\} \in \doubl(Y)$.
Except for the case where the cardinality of $X$ or $Y$ is equal to $2$, the homeomorphism $\phi$ can be chosen uniquely.
\end{proposition}

\begin{proof}
We will prove the case that $\metr(X) = \lpm_k(X,d_X)$ and $\metr(Y) = \lpm_k(Y,d_Y)$.
Let $\Phi : \doubl(Y) \to \doubl(X)$ be a bijection as in Proposition~\ref{eq.}.
It remains to show that there is a homeomorphism $\phi : Y \to X$ such that $\Phi(\{x,y\}) = \{\phi(x),\phi(y)\}$ for every doubleton $\{x,y\} \in \doubl(Y)$ and $\Phi^{-1}(\{z,w\}) = \{\phi^{-1}(z),\phi^{-1}(w)\}$ for every doubleton $\{z,w\} \in \doubl(X)$.
Remark that the cardinality of $X$ is coincident with the one of $Y$,
 and that the case where their cardinality is equal to $2$ is obviously valid.
Assume that the cardinality of $X$ and $Y$ is greater than $2$.
For each $y \in Y$, we will verify that the intersection $\bigcap_{z \in Y \setminus \{y\}} \Phi(\{y,z\})$ is a singleton.
Fixing distinct points $z_1$ and $z_2$ in $Y \setminus \{y\}$ arbitrarily, since $\Phi : \doubl(Y) \to \doubl(X)$ is injective,
 we have that $\Phi(\{y,z_1\}) \cap \Phi(\{y,z_2\})$ is empty or degenerate.
Supposing that $\Phi(\{y,z_1\})$ does not meet $\Phi(\{y,z_2\})$, we can find a pseudometric $d$ on the finite subset
 $$A = \Phi(\{y,z_1\}) \cup \Phi(\{y,z_2\}) \cup \Phi(\{z_1,z_2\})$$
 such that $d(\Phi(\{y,z_1\})) = d(\Phi(\{y,z_2\})) = 1$ and $d(\Phi(\{z_1,z_2\})) = 3$.
Take a sufficient large number $t \geq 1$ so that $d/t \in \lpm_k(A,d_X|_{A^2})$,
 which can be extended to $\lpm_k(X,d_X)$ by Theorem~\ref{ext.}.
According to Proposition~\ref{eq.},
\begin{align*}
 \frac{3}{t} &= \frac{d(\Phi(\{z_1,z_2\}))}{t} = T\bigg(\frac{d}{t}\bigg)(z_1,z_2)\\
 &\leq T\bigg(\frac{d}{t}\bigg)(y,z_1) + T\bigg(\frac{d}{t}\bigg)(y,z_2) = \frac{d(\Phi(\{y,z_1\}))}{t} + \frac{d(\Phi(\{y,z_2\}))}{t} = \frac{2}{t},
\end{align*}
 which is a contradiction.
Therefore $\Phi(\{y,z_1\}) \cap \Phi(\{y,z_2\}) = \{x\}$ for some point $x \in X$.
Similarly, suppose that there exists $z \in Y \setminus \{y\}$ such that $\Phi(\{y,z\})$ does not contain the point $x$,
 so we can deduce a contradiction.
Thus
 $$\{x\} = \bigcap_{z \in Y \setminus \{y\}} \Phi(\{y,z\}),$$
 so we can define a map $\phi : Y \to X$ by $\phi(y) = x$.
Similarly, we can obtain a map $\psi : X \to Y$ for the inverse $\Phi^{-1}$ satisfying that
 $$\{\psi(x)\} = \bigcap_{z \in X \setminus \{x\}} \Phi^{-1}(\{x,z\}).$$
As is easily observed,
 $\phi$ is a bijection and $\psi = \phi^{-1}$,
 and moreover, $\Phi(\{x,y\}) = \{\phi(x),\phi(y)\}$ and $\Phi^{-1}(\{z,w\}) = \{\phi^{-1}(z),\phi^{-1}(w)\}$.

We shall verify that $\phi$ is a homeomorphism.
For any point $y \in Y$ and any open neighborhood $U$ of $\phi(y)$ in $X$, there is a pseudometric $d \in \pseudo(\{\phi(y)\} \cup (X \setminus U))$ such that $d(\phi(y),x) = 2$ if $x \in X \setminus U$,
 and $d(x,x') = 0$ if $x, x' \in X \setminus U$.
Observe that $d$ is lipschitz with respect to $d_X|_{(\{\phi(y)\} \cup (X \setminus U))^2}$ such that
 $$\lipc(d) = \frac{2}{d_X(\phi(y),X \setminus U)} < \infty.$$
By Theorem~\ref{ext.}, $d$ can be extended over $X^2$ and $d/t \in \lpm_k(X,d_X)$ for some $t \geq 1$.
Due to Proposition~\ref{eq.}, for each point $z \in B_{T(d/t)}(y,1/t) \setminus \{y\}$,
 $$\frac{d(\phi(y),\phi(z))}{t} = \frac{d(\Phi(\{y,z\}))}{t} = T\bigg(\frac{d}{t}\bigg)(y,z) < \frac{1}{t},$$
 which means that $\phi(z) \in U$.
Hence $\phi$ is continuous.
Similarly, the inverse $\phi^{-1}$ is also continuous.

It remains to prove the uniqueness of $\phi$.
Take any homeomorphism $\psi : Y \to X$ such that for any $d \in \pseudo(X)$ and for any $x, y \in Y$, $T(d)(x,y) = d(\psi(x),\psi(y))$.
Assume that there is a point $x \in Y$ such that $\phi(x) \neq \psi(x)$.
For a finite set
 $$\{\phi(x),\phi(y),\psi(x),\psi(y)\} \subset X,$$
 where $y \in Y \setminus \{x,\phi^{-1}(\psi(x))\}$,
 choose a pseudometric $\rho$ on it such that $\rho(\phi(x),\phi(y)) = 0$ and $\rho(\psi(x),\psi(y)) = 1$.
Using Theorem~\ref{ext.}, we can extend $\rho$ to a lipschitz continuous pseudometric on the whole space $X$.
Then for sufficient large $t \geq 1$, $\rho/t \in \lpm_k(X,d_X)$,
 and
 $$0 = \frac{\rho(\phi(x),\phi(y))}{t} = T\bigg(\frac{\rho}{t}\bigg)(x,y) = \frac{\rho(\psi(x),\psi(y))}{t} = \frac{1}{t},$$
 which is a contradiction.
We conclude that $\phi = \psi$.
\end{proof}

\section{Proof of Theorems A, B and C}

In this section, we provide proofs of our Theorems A, B and C.

\begin{proposition}\label{oper.}
Let $\psi : Y \to X$ be a homeomorphism,
 and let $S : \pseudo(X) \to \pseudo(Y)$ be a map defined by for all $d \in \pseudo(X)$ and for all $x, y \in Y$,
 $$S(d)(x,y) = d(\psi(x),\psi(y)).$$
Then $S$ is an isometry satisfying that $S(\adm(X)) = \adm(Y)$ and $S({\rm P}(X)) = {\rm P}(Y)$.
Additionally, the following holds:
\begin{enumerate}
 \item If $\psi : (Y,\mathcal{U}_Y) \to (X,\mathcal{U}_X)$ is a uniformly homeomorphism,
 then $S(\upm(X,\mathcal{U}_X)) = \upm(Y,\mathcal{U}_Y)$ and $S({\rm U}(X,\mathcal{U}_X)) = {\rm U}(Y,\mathcal{U}_Y)$.
\end{enumerate}
Moreover, when $X = (X,d_X)$ and $Y = (Y,d_Y)$ are metric spaces,
 the following are valid:
\begin{enumerate}
\setcounter{enumi}{1}
 \item If $\psi : (Y,d_Y) \to (X,d_X)$ is bi-lipschitz,
 then $S(\lpm(X,d_X)) = \lpm(Y,d_Y)$.
 \item If $\psi : (Y,d_Y) \to (X,d_X)$ is an isometry and $k > 0$,
 then $S(\overline{\lpm}_k(X,d_X)) = \overline{\lpm}_k(Y,d_Y)$ and $S(\lpm_k(X,d_X)) = \lpm_k(Y,d_Y)$.
\end{enumerate}
\end{proposition}

\begin{proof}
The former holds by \cite[Lemma~2.4]{IK2} and the proof of Main Theorem of \cite{Kos29}.
First, we show (1) of the latter.
It follows from the uniform continuities $\psi$ and $\psi^{-1}$ that $S(\upm(X,\mathcal{U}_X)) = \upm(Y,\mathcal{U}_Y)$.
We will investigate that $S({\rm U}(X,\mathcal{U}_X)) = {\rm U}(Y,\mathcal{U}_Y)$.
Recall that $\psi$ is uniformly continuous with respect to $(\mathcal{U}_Y)_0$ and $(\mathcal{U}_X)_0$.
Indeed, by Proposition~\ref{base}, for each $U \in (\mathcal{U}_X)_0$, there are a finite set $F$ in $\uc(X,\mathcal{U}_X)$ and a positive integer $i \in \N$ such that if $d(F)(x,y) < 2^{-i}$,
 then $(x,y) \in U$.
Let
 $$F \circ \psi = \{f \circ \psi \mid f \in F\} \subset \uc(Y,\mathcal{U}_Y) \text{ and } B = \{(z,w) \in Y^2 \mid d(F \circ \psi)(z,w) < 2^{-i}\}.$$
Observe that $B \in \mathcal{B}_Y \subset (\mathcal{U}_Y)_0$ due to Proposition~\ref{base},
 and that for any $(z,w) \in B$,
 $$d(F)(\psi(z),\psi(w)) = \max\{|f(\psi(z)) - f(\psi(w))| \mid f \in F\} = d(F \circ \psi)(z,w) < 2^{-i},$$
 which implies that $(\psi(z),\psi(w)) \in U$.
Therefore $\psi$ is uniformly continuous with respect to $(\mathcal{U}_Y)_0$ and $(\mathcal{U}_X)_0$.
Fix every $d \in {\rm U}(X,\mathcal{U}_X)$.
Since $d$ is uniformly continuous with respect to $(\mathcal{U}_X)_0$ by Proposition~\ref{unif.} and $\psi$ is uniformly continuous with respect to $(\mathcal{U}_Y)_0$ and $(\mathcal{U}_X)_0$,
 $S(d)$ is uniformly continuous with respect to $(\mathcal{U}_Y)_0$ according to Proposition~8.1.22 of \cite{En},
 that is $S(d) \in {\rm U}(Y,\mathcal{U}_Y)$ by Proposition~\ref{unif.}.
We conclude that $S({\rm U}(X,\mathcal{U}_X)) \subset {\rm U}(Y,\mathcal{U}_Y)$.
By the same way, $S^{-1}({\rm U}(Y,\mathcal{U}_Y)) \subset {\rm U}(X,\mathcal{U}_X)$.

Next, we verify that $S(\lpm(X,d_X)) = \lpm(Y,d_Y)$ in (2).
Since $\psi$ is bi-lipschitz continuous with respect to $d_Y$ and $d_X$,
 there exists $\lambda \geq 1$ such that for all points $x, y \in Y$,
 $$\frac{1}{\lambda} d_Y(x,y) \leq d_X(\psi(x),\psi(y)) \leq \lambda d_Y(x,y).$$
Then for any $d \in \lpm(X,d_X)$ and any $x, y \in Y$ with $x \neq y$,
 $$\frac{S(d)(x,y)}{d_Y(x,y)} = \frac{d(\psi(x),\psi(y))}{d_Y(x,y)} \leq \frac{\lambda d(\psi(x),\psi(y))}{d_X(\psi(x),\psi(y))} \leq \lambda\lipc(d) < \infty.$$
Hence $S(\lpm(X,d_X)) \subset \lpm(Y,d_Y)$.
It follows from the same argument that $S^{-1}(\lpm(Y,d_Y)) \subset \lpm(X,d_X)$.

Third, in the case of (3), since $\psi$ is an isometry,
 for each metric $d \in \overline{\lpm}_k(X,d_X)$,
 \begin{align*}
  \lipc(S(d)) &= \sup_{x, y \in Y \text{ with } x \neq y} \frac{S(d)(x,y)}{d_Y(x,y)} = \sup_{x, y \in Y \text{ with } x \neq y} \frac{d(\psi(x),\psi(y))}{d_Y(x,y)}\\
  &= \sup_{x, y \in Y \text{ with } x \neq y} \frac{d(\psi(x),\psi(y))}{d_X(\psi(x),\psi(y))} = \lipc(d).
 \end{align*}
As a consequence, $S(\overline{\lpm}_k(X,d_X)) = \overline{\lpm}_k(Y,d_Y)$ and $S(\lpm_k(X,d_X)) = \lpm_k(Y,d_Y)$.
\end{proof}

Now we shall show Theorems~A, B and C.

\begin{proof}[Proof of Theorem~A]
First, the implication (1) $\Rightarrow$ (2) follows from (1) of Proposition~\ref{oper.}.
We will show the implication (2) $\Rightarrow$ (1).
It is sufficient to prove that a homeomorphism $\phi : Y \to X$ as in Proposition~\ref{corresp.} is uniformly homeomorphic.
For each $U \in \mathcal{U}_X$, there exists $d \in \upm(X,\mathcal{U}_X)$ such that if $d(x,y) < 1$,
 then $(x,y) \in U$ due to \cite[Corollary~8.1.11]{En}.
Then $T(d) \in \upm(Y,\mathcal{U}_Y)$ and we can choose $V \in \mathcal{U}_Y$ so that for any $(z,w) \in V$, $T(d)(z,w) < 1$,
 and hence by Proposition~\ref{corresp.},
 $$d(\phi(z),\phi(w)) = T(d)(z,w) < 1.$$
Therefore $(\phi(z),\phi(w)) \in U$,
 so $\phi$ is uniformly continuous.
Similarly, $\phi^{-1}$ is also uniformly continuous.

In the latter case, (1) $\Rightarrow$ (3) also follows from (1) of Proposition~\ref{oper.}.
To show the implication (3) $\Rightarrow$ (1), take the Smirnov compactifications $uX = (\overline{X},\overline{(\mathcal{U}_{d_X})_0})$ and $uY = (\overline{Y},\overline{(\mathcal{U}_{d_Y})_0})$ of $X$ and $Y$, respectively.
By virtue of \cite[Lemma~1.4]{Isb1} (cf.~Problem~8.5.6(a) of \cite{En}), we can obtain the extension $e : {\rm U}(X,d_X) \to \upm(\overline{X},\overline{(\mathcal{U}_{d_X})_0})$.
Since $X$ is dense in $uX$,
 the map $e$ is an isometric embedding and preserves the cone structure.
Moreover, for each $d \in \upm(\overline{X},\overline{(\mathcal{U}_{d_X})_0})$, the restriction $d|_{X^2} \in {\rm U}(X,d_X)$ by Proposition~\ref{unif.},
 and hence we can define the inverse $e^{-1}$ of $e$ by $e^{-1}(d) = d|_{X^2}$.
In conclusion, we can identify
 $${\rm U}(X,d_X) = \upm(\overline{X},\overline{(\mathcal{U}_{d_X})_0}) = \pseudo(uX) \text{ and } {\rm U}(Y,d_Y) = \upm(\overline{Y},\overline{(\mathcal{U}_{d_Y})_0}) = \pseudo(uY).\footnote{On the other hand, it is known that for every metric space $Z = (Z,d_Z)$, $\uc(Z,d_Z) = \conti(uZ)$, see \cite[Theorem~2.5]{Woo}.}$$
Since $T$ is identified with the isometry from $\pseudo(uX)$ to $\pseudo(uY)$ and the both of $uX$ and $uY$ are compact,
 we can obtain the uniform homeomorphism $\phi : uY \to uX$ as in Proposition~\ref{corresp.} by the same way as the above implication (2) $\Rightarrow$ (1).
Following the same argument as \cite{Her} (cf.~\cite{Ef} and \cite[Problems~8.5.19.(a) and 8.5.20]{En}), we have that the restriction $\phi|_{Y} : (Y,d_Y) \to (X,d_X)$ is a uniformly homeomorphism.
Furthermore, the canonical formula of $T$ and the uniqueness of $\phi$ hold by Proposition~\ref{corresp.}.
\end{proof}

\begin{proof}[Proof of Theorem~B]
From (2) of Proposition~\ref{oper.}, (1) $\Rightarrow$ (2) holds.
To prove (2) $\Rightarrow$ (1), letting $\phi : Y \to X$ be a homeomorphism as in Proposition~\ref{corresp.} and
 $$\lambda = \max\{\lipc(T(d_X)),\lipc(T^{-1}(d_Y))\},$$
 we have that for any $z, w \in Y$ with $z \neq w$,
 $$\frac{d_X(\phi(z),\phi(w))}{d_Y(z,w)} = \frac{T(d_X)(z,w)}{d_Y(z,w)} \leq \lipc(T(d_X)) \leq \lambda,$$
 and that for any $x, y \in X$ with $x \neq y$,
 $$\frac{d_Y(\phi^{-1}(x),\phi^{-1}(y))}{d_X(x,y)} = \frac{T^{-1}(d_Y)(x,y)}{d_X(x,y)} \leq \lipc(T^{-1}(d_Y)) \leq \lambda.$$
Therefore for every pair $(x,y) \in Y^2$,
 $$\frac{1}{\lambda} d_Y(x,y) \leq d_X(\phi(x),\phi(y)) \leq \lambda d_Y(x,y),$$
 which means that $\phi$ is bi-lipschitz.
By virtue of Proposition~\ref{corresp.}, the canonical formula of $T$ and the uniqueness of $\phi$ are valid.
\end{proof}

\begin{proof}[Proof of Theorem~C]
The implication (1) $\Rightarrow$ (2) follows from (3) of Proposition~\ref{oper.},
 and the implication (2) $\Rightarrow$ (3) is obvious.
To show (3) $\Rightarrow$ (1), suppose that $\phi : Y \to X$ is a homeomorphism as in Proposition~\ref{corresp.}.
Then since $k'd_X$ and $k'd_Y$ are $k'$-lipschitz for each positive number $k' < k$,
 observe that for all $z, w \in Y$ with $z \neq w$,
 $$\frac{k'd_X(\phi(z),\phi(w))}{d_Y(z,w)} = \frac{T(k'd_X)(z,w)}{d_Y(z,w)} \leq \lipc(T(k'd_X)) < k,$$
 and that for all $x, y \in X$ with $x \neq y$,
 $$\frac{k'd_Y(\phi^{-1}(x),\phi^{-1}(y))}{d_X(x,y)} = \frac{T^{-1}(k'd_Y)(x,y)}{d_X(x,y)} \leq \lipc(T^{-1}(k'd_Y)) < k.$$
By the arbitrariness of $k'$, for all points $x, y \in Y$,
 $$d_Y(x,y) \leq d_X(\phi(x),\phi(y)) \leq d_Y(x,y),$$
 so $d_X(\phi(x),\phi(y)) = d_Y(x,y)$.
Consequently, $\phi$ is an isometry.
The canonical formula of $T$ and the uniqueness of $\phi$ follow from Proposition~\ref{corresp.} immediately.
\end{proof}

\section{Spaces of admissible metrics}

This section is devoted to consider Banach-Stone type theorems on spaces of admissible metrics.
Similar to (3) and (4) of Theorem~\ref{isometry}, if $X = (X,d_X)$ and $Y = (Y,d_Y)$ are metric spaces,
 then the following are equivalent to the conditions in Theorem~A:
 \begin{itemize}
  \item[($\star$)] there exists an isometry $T : \upm(X,d_X) \cap \adm(X) \to \upm(Y,d_Y) \cap \adm(Y)$ with $T({\rm P}(X) \cap \upm(X,d_X) \cap \adm(X)) = {\rm P}(Y) \cap \upm(Y,d_X) \cap \adm(Y)$.
  \item[($\star\star$)] there is an isometry $T : {\rm P}(X) \cap \upm(X,d_X) \cap \adm(X) \to {\rm P}(Y) \cap \upm(Y,d_X) \cap \adm(Y)$.
 \end{itemize}
In fact, by virtue of Proposition~\ref{oper.}, we can show that if $X = (X,d_X)$ and $Y = (Y,d_Y)$ are uniformly isomorphic,
 then ($\star$) and ($\star\star$) hold.

When $Z = (Z,d_Z)$ is a metric space,
 the metric $e = \min\{d_Z/d_Z(x,y),1\}$ in Lemma~\ref{transl.} is admissible,
 and hence so are the sums $\rho = e + \sum_{n = 1}^\infty \rho_n/2^n$ and $d + \rho$ by \cite[Lemma~2.1]{IK2}.
This guarantees that the argument in Section~\ref{peak} is valid if ${\rm P}(X)$ and ${\rm P}(Y)$ are replaced with ${\rm P}(X) \cap \adm(X)$ and ${\rm P}(Y) \cap \adm(Y)$, respectively.
Therefore the following is equivalent to the conditions in Theorem~A:
 \begin{itemize}
  \item[(2)'] there is an isometry $T : \upm(X,d_X) \to \upm(Y,d_Y)$ with $T({\rm P}(X) \cap \upm(X,d_X) \cap \adm(X)) = {\rm P}(Y) \cap \upm(Y,d_X) \cap \adm(Y)$.
 \end{itemize}

For every metric space $Z = (Z,d_Z)$, $\adm(Z)$ is a dense subset of $\pseudo(Z)$ (see \cite[Proposition~5]{Kos20}).
Similar to \cite[Proposition~2.4]{Kos29}, we can prove that the subset ${\rm P}(Z) \cap \upm(Z,d_Z) \cap \adm(Z)$ is dense in the complete metric space $\upm(Z,d_Z)$.
Hence the isometries $T : \upm(X,d_X) \cap \adm(X) \to \upm(Y,d_Y) \cap \adm(Y)$ with $T({\rm P}(X) \cap \upm(X,d_X) \cap \adm(X)) = {\rm P}(Y) \cap \upm(Y,d_X) \cap \adm(Y)$ in ($\star$) and $T : {\rm P}(X) \cap \upm(X,d_X) \cap \adm(X) \to {\rm P}(Y) \cap \upm(Y,d_X) \cap \adm(Y)$ in ($\star\star$) are extendable to the isometry $T$ of (2)'.
Consequently, the both of implications ($\star$) $\Rightarrow$ (2)' and ($\star\star$) $\Rightarrow$ (2)' hold.

Furthermore, the following is equivalent to the conditions in Theorem~C:
 \begin{itemize}
  \item[($\star\star\star$)] there exists an isometry $T : \lpm_k(X,d_X) \cap \adm(X) \to \lpm_k(Y,d_Y) \cap \adm(Y)$ such that $T({\rm P}(X) \cap \lpm_k(X,d_X) \cap \adm(X)) = {\rm P}(Y) \cap \lpm_k(Y,d_X) \cap \adm(Y)$,
  and for each $d \in \lpm_k(X,d_X) \cap \adm(X)$ and each $t \in (0,1)$, $T(td) = tT(d)$.
 \end{itemize}
Indeed, it follows from Proposition~\ref{oper.} that the implication (1) $\Rightarrow$ ($\star\star\star$) holds.
Recall that since the isometry as in Proposition~\ref{oper.} is a composition operator by a homeomorphism,
 it is scalar-preserving.
By the same argument as the above, the following is equivalent to the conditions in Theorem~C:
 \begin{itemize}
  \item[(2)''] there exists an isometry $T : \overline{\lpm}_k(X,d_X) \to \overline{\lpm}_k(Y,d_Y)$ with $T(\lpm_k(X,d_X)) = \lpm_k(Y,d_Y)$ and $T({\rm P}(X) \cap \lpm_k(X,d_X) \cap \adm(X)) = {\rm P}(Y) \cap \lpm_k(Y,d_Y) \cap \adm(Y)$.
 \end{itemize}

For every metric space $Z = (Z,d_Z)$, we will verify that $\lpm_k(Z,d_Z) \cap \adm(Z)$ is dense in $\overline{\lpm}_k(Z,d_Z)$.
As is observed in Section~\ref{unif.metr.},
 $\lpm_k(Z,d_Z) $ is a dense set in $\overline{\lpm}_k(Z,d_Z)$.
Thus it remains to check that $\lpm_k(Z,d_Z) \cap \adm(Z)$ is dense in $\lpm_k(Z,d_Z)$.
To prove it, fix any $d \in \lpm_k(Z,d_Z)$ and any $\epsilon > 0$.
Let
 $$\delta = \frac{\min\{\epsilon/\|d_Z\|,k - \lipc(d)\}}{2},$$
 so due to \cite[Lemma~2.1]{IK2}, $d + \delta d_Z \in \adm(Z)$.
For all distinct points $x, y \in Z$,
 $$\frac{d(x,y) + \delta d_Z(x,y)}{d_Z(x,y)} = \frac{d(x,y)}{d_Z(x,y)} + \delta \leq \lipc(d) + \frac{k - \lipc(d)}{2} = \frac{k + \lipc(d)}{2} < k,$$
 which implies that $d + \delta d_Z \in \lpm_k(Z,d_Z)$.
Moreover,
 $$\|d - (d + \delta d_Z)\| = \|\delta d_Z\| = \delta\|d_Z\| \leq \frac{\epsilon}{\|d_Z\|} \cdot \|d_Z\| = \epsilon.$$
It follows that $\lpm_k(Z,d_Z) \cap \adm(Z)$ is dense in a complete metric space $\overline{\lpm}_k(Z,d_Z)$.

Hence we can extend the isometry $T : \lpm_k(X,d_X) \cap \adm(X) \to \lpm_k(Y,d_Y) \cap \adm(Y)$ in ($\star\star\star$) to an isometry $\overline{T} : \overline{\lpm}_k(X,d_X) \to \overline{\lpm}_k(Y,d_Y)$.
Then from the density of $\lpm_k(X,d_X) \cap \adm(X) \subset \overline{\lpm}_k(X,d_X)$, we also deduce that for every $d \in \overline{\lpm}_k(X,d_X)$ and every $0 < t < 1$, $\overline{T}(td) = t\overline{T}(d)$.
Moreover, when $t \geq 1$ and $td \in \overline{\lpm}_k(X,d_X)$,
 $$\frac{1}{t}\overline{T}(td) = \overline{T}(\frac{1}{t} \cdot td) = \overline{T}(d),$$
 that is $\overline{T}(td) = t\overline{T}(d)$.
We shall show that $\overline{T}(\lpm_k(X,d_X)) = \lpm_k(Y,d_Y)$.
For any pseudometric $d \in \lpm_k(X,d_X)$, we can choose a positive number $t > 1$ so that $t\lipc(d) \leq k$.
Then $td \in \overline{\lpm}_k(X,d_X)$ and $\overline{T}(td) \in \overline{\lpm}_k(Y,d_Y)$.
Observe that
 $$t\lipc(\overline{T}(d)) = \lipc(t\overline{T}(d)) = \lipc(\overline{T}(td)) \leq k,$$
 so $\lipc(\overline{T}(d)) \leq k/t < k$.
Therefore $\overline{T}(\lpm_k(X,d_X)) \subset \lpm_k(Y,d_Y)$.
Similarly, $\overline{T}(\lpm_k(X,d_X)) \supset \lpm_k(Y,d_Y)$.
We conclude that ($\star\star\star$) $\Rightarrow$ (2)'' is valid.

\section{Problem}

Given a metric space $Z = (Z,d_Z)$, we consider the space $\lip(Z,d_Z)$ of real-valued bounded lipschitz continuous functions on $Z$ endowed with the following norm:
 $$\|f\|_{lip} = \max\{\|f\|,\lipc(f)\},$$
 where the value
 $$\lipc(f) = \sup_{x, y \in Z \text{ with } x \neq y} \frac{|f(x) - f(y)|}{d_Z(x,y)}$$
 is the lipschitz constant.
N.~Weaver \cite{Wea} (cf.~\cite{Ro,Vas}) established the following Banach-Stone type theorem on it.

\begin{theorem}
Assume that $X = (X,d_X)$ and $Y = (Y,d_Y)$ are $1$-connected\footnote{A metric space $Z = (Z,d_Z)$ is \textit{$1$-connected} provided that it can not be decomposed into two disjoint subsets whose distance is greater than or equal to $1$.} complete metric spaces of diameters less than or equal to $2$.
The following are equivalent:
\begin{enumerate}
 \item $X$ and $Y$ are isometric;
 \item there exists a linear isometry $T : \lip(X,d_X) \to \lip(Y,d_Y)$.
\end{enumerate}
In this case, for each isometry $T : \lip(X,d_X) \to \lip(Y,d_Y)$, there exists an isometry $\phi : Y \to X$ and $\alpha \in \conti(Y)$ with $\alpha(Y) \subset \{1,-1\}$ such that for any $f \in \lip(X,d_X)$ and for any $y \in Y$,
 $$T(f)(y) = \alpha(y)f(\phi(y)).$$
\end{theorem}

As an analogue of the above theorem, can we obtain Banach-Stone type theorems on the space $\lpm(Z,d_Z) = (\lpm(Z,d_Z),\|\cdot\|_{lip})$ and its subspaces?


\begin{thebibliography}{99}
 \bibitem{AF} J.~Araujo and J.J,~Font, \textit{Linear isometries on subalgebras of uniformly continuous functions}, Proc. Edinburgh Math. Soc. (2) \textbf{43} (2000), no.~1, 139--147.
 \bibitem{Banac} S.~Banach, \textit{Th\'{e}orie des op\'{e}rations lin\'{e}aires}, Chelsea, Warsaw, (1932).
 \bibitem{BBST} T.~Banakh, N.~Brodskiy, I.~Stasyuk and E.D.~Tymchatyn, \textit{On continuous extension of uniformly continuous functions and metrics}, Colloq. Math. \textbf{116} (2009), no.~2, 191--202.
 \bibitem{CJV} M.G.~Cabrera-Padilla, A.~Jim\'{e}nez-Vargas and M.~Villegas-Vallecillos, \textit{A survey on isometries between Lipschitz spaces}, Associative and non-associative algebras and applications, Springer Proc. Math. Stat., \textbf{311}, Springer, Cham, (2020), 51--88.
 \bibitem{Ef} V.A.~Efremovi\v{c}, \textit{The geometry of proximity. I}, Mat. Sbornik N.S. \textbf{31/73} (1952), 189--200.
 \bibitem{En} R.~Engelking, General Topology, Revised and Complete Edition, Sigma Ser. in Pure Math., \textbf{6}, Heldermann Verlag, Berlin, (1989).
 \bibitem{FJ} R.J.~Fleming and J.E.~Jamison, Isometries on Banach spaces: function spaces, Chapman \& Hall/CRC Monogr. Surv. Pure Appl. Math., \textbf{129}, Boca Raton, FL, (2003).
 \bibitem{GJ1} M.I.~Garrido and J.~Jaramillo, \textit{Variations on the Banach-Stone theorem}, Extracta Math. \textbf{17} (2002), no.~3, 351--383.
 \bibitem{Hau} F.~Hausdorff, \textit{Erweiterung einer Hom\"{o}omorphie}, Fund. Math. \textbf{16} (1930), 353--360.
 \bibitem{Her} S.~Hern\'{a}ndez, \textit{Uniformly continuous mappings defined by isometries of spaces of bounded uniformly continuous functions}, Houston J. Math. \textbf{29} (2003), no.~1, 149--155.
 \bibitem{HMM} D.~Hirota, I.~Matsuzaki and T.~Miura, \textit{Phase-isometries between the positive cones of the Banach spaces of continuous real-valued functions}, Ann. Funct. Anal. \textbf{15} (2024), no.~4, Paper No.~77, 11pp.
 \bibitem{IK2} Y.~Ishiki and K.~Koshino, \textit{On isometric universality of spaces of metrics}, Topology Appl. \textbf{369} (2025), Paper No.~109394, 15 pp.
 \bibitem{Isb1} J.R.~Isbell, \textit{On finite-dimensional uniform spaces}, Pacific J. Math. \textbf{9} (1959), 107--121.
 \bibitem{Kos20} K.~Koshino, \textit{Recognizing the topologies of spaces of metrics with the topology of uniform convergence}, Bull. Pol. Acad. Sci. Math. \textbf{70} (2022), 165--171.
 \bibitem{Kos29} K.~Koshino, \textit{Isometries between spaces of metrics}, Adv. Oper. Theory \textbf{10} (2025), 83, 16pp.
 \bibitem{MU} S.~Mazur and S. Ulam, \textit{Sur les transformation d'spaces vectoriels norm\'{e}}, C.R. Acad. Sci. Paris \textbf{194} (1932), 946--948.
 \bibitem{Ro} A.K.~Roy, \textit{Extreme points and linear isometries of the Banach space of Lipschitz functions}, Canadian J. Math. \textbf{20} (1968), 1150--1164.
 \bibitem{Sh} M.E.~Shanks, \textit{The space of metrics on a compact metrizable space}, Amer. J. Math. \textbf{66} (1944), 461--469.
 \bibitem{MHSto} M.H.~Stone, \textit{Applications of the theory of Boolean rings to general topology}, Trans. Amer. Math. Soc. \textbf{41} (1937), no.~3, 375--481.
 \bibitem{Vas} M.H.~Vasavada, \textit{Closed Ideals and Linear Isometries of Certain Function Spaces},
Thesis (Ph.D.) The University of Wisconsin - Madison, ProQuest LLC, Ann Arbor, MI, (1969), 92 pp.
 \bibitem{Wea} N.~Weaver, \textit{Isometries of noncompact Lipschitz spaces}, Canad. Math. Bull. \textbf{38} (1995), no.~2, 242--249.
 \bibitem{Woo} R.G.~Woods, \textit{The minimum uniform compactification of a metric space}, Fund. Math. \textbf{147} (1995), no.~1, 39--59.
\end{thebibliography}
\end{document}